\documentclass[10pt]{article}%

\usepackage{amsfonts,bm}
\usepackage{amsmath}
\usepackage{amssymb}
\usepackage{amsbsy}
\usepackage{amsthm}
\usepackage{epsfig}
\usepackage{epstopdf}
\usepackage{graphicx}
\usepackage{colordvi}
\usepackage{graphics}
\usepackage{subfigure}
\usepackage{color}
\usepackage{fullpage}
\usepackage{comment}
\usepackage{float}
\usepackage{enumerate}
\usepackage{adjustbox}
\usepackage{tabularx}
\usepackage{epstopdf}
\usepackage[skip=2pt]{caption}
\usepackage[labelfont=bf]{caption}						
\usepackage{booktabs,array}					
\usepackage{rotating}

\usepackage{multirow}
\usepackage{bigstrut}
\usepackage{rotating}
\usepackage{makecell}
\newcolumntype{P}[1]{>{\raggedright}p{#1}}
\newcolumntype{M}[1]{>{\raggedright}m{#1}}

\numberwithin{equation}{section}
\newcommand{\norm}[1]{\left\Vert#1\right\Vert}
\newtheorem{theorem}{Theorem}[section]

\newtheorem{lemma}{Lemma}[section]

\newtheorem{remark}{Remark}[section]

\newtheorem{algorithm}{Algorithm}[section]

\newcommand{\bu}{{u}}
\newcommand{\bv}{{v}}
\newcommand{\bw}{{w}}
\newcommand{\bx}{{x}}

\newcommand{\be}{{e}}

\newcommand{\bvarphi}{{\boldsymbol \varphi}}
\newcommand{\bX}{{X}}
\newcommand{\bV}{{V}}

\newcommand{\bphi}{{\boldsymbol \phi}}
\newcommand{\bfeta}{{\boldsymbol \eta}}

\newcommand{\bLambda}{{\boldsymbol \Lambda}}

\def \bX{{\mathbf X}}
\def \bV{{\mathbf V}}
\def \bH{{\mathbf H}}

\def \bu{{\mathbf u}}
\def \bv{{\mathbf v}}
\def \bV{{\mathbf V}}
\def \bw{{\mathbf w}}
\def \be{{\mathbf e}}

\def \bf{{\mathbf f}}
\def \b0{{\mathbf 0}}
\def \bx{{\mathbf x}}
\def \bP{{\mathbf P}}

\usepackage{caption}

\begin{document}
\title{Modular grad-div stabilization for the incompressible non-isothermal fluid flows}

\author{Mine Akbas \footnote{Department of Mathematics, Duzce University, 81620, Duzce, Turkey; mineakbas@duzce.edu.tr. } 
\hspace{0.2in}
Leo G. Rebholz\footnote{Department of Mathematical Sciences, Clemson University, Clemson, SC 29634; rebholz@clemson.edu;
This work was partially supported by NSF grant DMS1522191.}
\hspace{.2in}
}

\date{}
\maketitle

\begin{abstract}
This paper considers a modular grad-div stabilization method for approximating solutions of the time-dependent Boussinesq model of non-isothermal flows. The proposed method adds a minimally intrusive step to an existing Boussinesq code, with the key idea being that the penalization of the divergence errors, is only in the extra step (i.e. nothing is added to the original equations). The paper provides a full mathematical analysis by proving unconditional stability and optimal convergence of the methods considered. Numerical experiments confirm theoretical findings, and show that the algorithms have a similar positive effect as the usual grad-div stabilization. 
\end{abstract} 
\section{Introduction}
Classical conforming finite element discretizations for incompressible flows relax the divergence constraint, and enforce it only weakly. While this enables one to construct inf-sup stable discretizations, weak enforcement leads to errors depending on the continuous pressure scaled by the Reynolds number, and creates inaccurate computed solutions for many flow problems, including Boussinesq flows \cite{GLCL80, PFC89, DGT94, bercovier2, F98, GLRW12},
potential and generalized Beltrami flows \cite{LM16, LM162, JLMNR17}, quasi-geostrophic flows \cite{TC12, CT14, LM16}, and two-phase flows with surface tension \cite{GMT07, LMT16}. \\ 
\textcolor{red}{Techniques to overcome this issue include using divergence-free elements or grad-div stabilization.  Using divergence-free elements, such as Scott-Vogelius elements (see \cite{GS19,JLMNR17} and references therein) eliminates the effect of the continuous pressure on the velocity error.  However, using divergence-free elements, in particular on quadrilateral meshes and/or in 3D, may be difficult to implement due to mesh restrictions, high polynomial degrees, and not being built into most major finite element software packages.  Moreover, in legacy codes, such changes may be impossible without a full rewrite. Grad-div stabilization has been recently studied from both theoretical and computational points of view \cite{OR04, OLHL09, LRW11, JJLR14, BR15}, and the studies show that it improves the accuracy of the approximate solutions for the Stokes/Navier-Stokes and related coupled multiphysics problems by reducing the effect of the continuous pressure on the velocity error \cite{DGT94, TV96, O02, OR04, LMNOR09, OLHL09, JK10, MNOR11, GLRW12}. While easier to implement in legacy codes compared to changing to divergence-free elements, there are also disadvantages: this stabilization increases coupling in the linear system, and leads to linear algebraic systems often more difficult to solve since the matrix contribution to the velocity block is singular.}\\  
Recently, a variant of grad-div stabilization was introduced for the incompressible NSE in \cite{FLR18}, which is more attractive from an implementation standpoint. The proposed algorithm in \cite{FLR18} adds a minimally intrusive module which is used after each time step in a Navier-Stokes solver. This extra step implements the first order grad-div stabilization separately, and penalizes the divergence of the velocity error, both in $L^2$ and $L^\infty$-norms. Hence, the algorithm retains benefits of classical grad-div stabilization, but adds resistance to solver breakdown as the stabilization parameters increase. The application of the grad-div step for any multistep time discretization can be found in \cite{F18} where the numerical scheme which uses second order grad-div step was analyzed and performed for the NSE.\\
The purpose of this paper is to extend these novel ideas from \cite{FLR18} to incompressible non-isothermal fluid flows governed by the Boussinesq equations. The numerical scheme for the Boussinesq equations consists of two steps. The first step approximates the usual Boussinesq equations with the backward Euler temporal and finite element spatial discretizations. The second step is a post processing step, and introduces a decoupled, first order grad-div stabilization step for the velocity. The novelty of this algorithm is that grad-div step is decoupled from evolution equations, and hence it can be easily used with an existing code. Moreover, since the grad-div step is separate, the penalization can happen without the negative effects on the saddle point system which can occur when parameters are bigger than one.  This paper studies the stability and convergence properties of the proposed method, and provides numerical experiments to illustrate its reliability and effectiveness.\\ 
This paper is arranged as follows. Section 2 gathers necessary notation and mathematical preliminaries. Section 3 introduces a modular grad-div stabilization method for the incompressible Boussinesq equations, and studies its stability and convergence. It also presents some numerical experiments to test the effectiveness and reliability of the method. The last section summarizes the results of the paper. 
\section{Mathematical Preliminaries} 
This section introduces mathematical preliminaries and notation. We assume that $\Omega$ in $\mathbb{R}^d \, (d=2,3)$ is a polygonal or polyhedral domain with the boundary $\partial \Omega$. Standard notation of Lebesgue and Sobolev spaces are used throughout this paper. The inner product in $(L^2(\Omega))^d$ is denoted by $(\cdot,\cdot)$, the norm in
$(L^2(\Omega))^d$ by $\|\cdot\|$ and the norm in the Hilbert space $(H^k(\Omega))^d$  by $\|\cdot\|_k$. For $X$ being a normed function space in $\Omega$, $L^p (0, T;X)$ is the space of all functions defined on $(0,T)\times\Omega$ for which the norm {is bounded}
$$\|u\|_{L^p(0,T;X)} :=\int\limits_{0}^T \|u\|_{X}^p \mathrm{d}\bx,\hspace{2mm}p\in [1,\infty).$$
For $p =\infty $, the usual modification is used in the definition of this space. 
We consider the classical function spaces
\begin{eqnarray*}
\bX: &=& (H_0^1(\Omega))^d :=\{\bv\in (L^2(\Omega))^d:\nabla \bv \in L^2(\Omega)^{d\times d}, \, \bv= \bm 0\,\hspace{2mm}\mbox{on}\,\partial \Omega\},
\nonumber
\\
Q:&=& L_0^2(\Omega) := \{q \in L^2(\Omega): \int_{\Omega} q \ dx = 0\}, \nonumber\\
W:&=& H_0^1(\Omega).
\end{eqnarray*}
For $\textbf{f}$ an element in the dual space of $\bX$, its norm is defined by
$$\|\textbf{f}\|_{-1}:=\sup\limits_{\bv\in \bX}\frac{| (\textbf{f}, \bv)|}{\|\bv\|_{L^2}}.$$
In this setting, we have the Poincar\'e-Friedrichs' inequality: $\forall v \in W$
$$
\norm{v}_{L^2} \leq C_P \norm{\nabla v}_{L^2},
$$
where $C_P$ is a constant depending only on the size of $\Omega$ \cite{Laytonbook}. We define the trilinear forms: $$ b(\bu,\bv,\bm w) := \frac{1}{2}\left( \left(\bu\cdot\nabla \bv, \bm w\right)-\left(\bu\cdot\nabla \bw, \bv\right) \right),\hspace{2mm}\forall \bu, \bv,\bw  \in \bX,$$
$$ b^*(\bu,\varphi,  \psi) := \frac{1}{2}\left( \left(\bu\cdot\nabla \varphi, \psi\right)-\left(\bu\cdot\nabla \psi, \varphi\right) \right),\hspace{2mm}\forall\bu\in\bX,\hspace{1mm} \text{and} \hspace{1mm}  \forall\varphi, \psi\in W.$$
The discrete time analysis needs the following norms: for $1 \leq k < \infty$
$$ \||v^{n}|\|_{\infty,k}:=\max\limits_{1\leq n \leq N}\|v^{n}\|_{k}, \hspace{1mm}\||v^{n}|\|_{p,k}:=\left( \Delta t\sum_{n=0}^{N-1}\|v^{n}\|_{k}^{p}\right)^{1/p}.$$
The following lemma is necessary to bound the trilinear terms in the  analysis.
\begin{lemma}\label{skewlemma}
There exists a constant $C$ such that for all $\bu,\bv,\bm w\in \bX$
\begin{align*}
b(\bu,\bv,\bm w)& \leq C \|\nabla \bu\|_{L^2}\|\nabla \bv\|_{L^2}\|\nabla \bm w\|_{L^2},\\
b(\bu,\bv,\bm w)& \leq C \sqrt{\|\bu\|_{L^2}\|\nabla \bu\|_{L^2}}\|\nabla \bv\|_{L^2}\|\nabla \bm w\|_{L^2}.
\end{align*}
\end{lemma}
\begin{proof}
Application of  H\"older's inequality, interpolation theorem, the Sobolev embedding theorem and
Poincar\'e-Friedrichs' inequality yield the result, see \cite{Laytonbook}.
\end{proof}
For a spatial discretization, we consider a conforming finite element spaces ${\bX}_h \subset \bX, Q_h\subset Q, Y_h\subset W$ defined on a regular triangulation $\mathcal{T}_h$ of the domain $\Omega$ with
maximum diameter $h$. For the stability of the pressure, $({\bX}_h , Q_h)$ is assumed to satisfy the discrete inf-sup condition: there is a constant $%
\alpha$ independent of the mesh size $h$ such that
\begin{eqnarray}
\inf_{q_h\in{Q}_h}\sup_{\bv_h\in {\bX}_h}\frac{(q_h,\,\nabla\cdot \bv_h)}{||\nabla
\bv_h\,||_{L^2}\,||\,q_h\,||_{_{L^2}}}\geq \alpha > 0.  \label{infsup}
\end{eqnarray}
We also assume that the finite element spaces $({\bX}_h,  Q_h, Y_h)$, satisfy approximation properties of piecewise polynomials of local degree $k, \,  k-1$, and $k$, respectively,
\begin{align}
\inf_{\bv_{h} \in \bX_{h}} \left \{\| {\bu-\bv_{h}}\|_{L^2} + h \| {\nabla
(\bu-\bv_{h})}\|_{L^2} \right\} &\leq C h^{k+1} \|{\bu}\|_{k+1},
\label{app1} \\
\inf_{q_{h} \in Q_{h}}\| {p-q_{h}}\|_{L^2} &\leq C h^{k}\|p\|_{k},  \label{app3} \\
\inf_{\theta_{h} \in {Y}_{h}} \left \{\| {\theta-\theta_{h}}\|_{_{L^2}} + h \| {\nabla
(\theta-\theta_{h})}\|_{L^2} \right\} &\leq C h^{k+1}\|\theta\|_{k+1}.
\label{app2}
\end{align}
The discretely divergence-free subspace of $\mathbf{X}_h$ is defined by:
\begin{equation*}
{\bV}_h := \{\bv_h\in {\bX}_h: (q_h,\nabla \cdot \bv_h) = 0, \forall
q_h\in Q_h\}.
\end{equation*}
It is known that under the inf-sup condition (\ref{infsup}), the discretely
divergence-free subspace $\mathbf{V}_h$ has the same approximation
properties as $\bX_h$ \cite{BS08}: 
$$\inf_{\bv_{h} \in \bV_{h}}  \| {\nabla
(\bu-\bv_{h})}\|_{L^2} \leq C({\alpha})\inf_{\bv_{h} \in \bX_{h}}\| {\nabla
(\bu-\bv_{h})}\|_{L^2} .$$ 
Our finite element analysis needs the standard inverse inequality: for any $\bv\in \bX_h,$ 
$$\|\nabla\bv\|_{L^2} \leq C_{inv}h^{-1}\|\bv\|_{L^2},$$
where $C_{inv}$ depends on the minimum angle in the triangulation.\\
\textcolor{red}{The important lemma necessary for our convergence analysis is Agmon's Inequality, which uses two interpolation inequalities between the Lebesgue space $L^{\infty}(\Omega)$ and the Sobolev spaces $H^2(\Omega)$ :
\begin{lemma}
Let $\Phi \in H^2(\Omega)\cap H_0^1(\Omega)$, $\Omega \subset \mathbb{R}^d, \,\, d=2,3$. Then there exists a constant $C$ such that 
\begin{align}
\|\Phi\|_{L^\infty}\leq C \|\Phi\|_{H^1}^{1/2}\|\Phi\|_{H^2}^{1/2}.
\end{align}
\end{lemma}
}
In our convergence analysis, we need a different version of the usual discrete Gronwall's Lemma in literature, see e.g.,\cite{HR90}:
\begin{lemma}[Discrete Gronwall's Lemma]
\label{DiscreteGronwall} Let $\Delta t$, $B$ and $a_n, b_n, c_n, d_n$ be
finite non-negative numbers such that
\begin{eqnarray*}
a_{N}+\Delta t\sum_{n=0}^{N}b_n\leq\Delta t\sum_{n=0}^{N-1}d_{n}a_{n}+\Delta
t\sum_{n=0}^{N}c_{n}+B \quad\mbox{for}\quad N\geq 1.
\end{eqnarray*}
Then for all $\Delta t>0$,
\begin{eqnarray*}
a_{N}+\Delta t\sum_{n=0}^{N}b_n\leq exp\bigg(\Delta t\sum_{n=0}^{N-1}d_{n}%
\bigg)\bigg(\Delta t\sum_{n=0}^{N}c_{n}+B\bigg)\quad\mbox{for}\quad N\geq 1.
\end{eqnarray*}
\end{lemma}
\section{First order modular grad-div stabilization for the Boussinesq equations}
This section presents a modular grad-div method based on backward-Euler time and finite element spatial discretizations for the incompressible Boussinesq equations, and gives its stability and convergence results.\\
Incompressible, non-isothermal fluid flows are governed by the incompressible Navier-Stokes equations (NSE) and heat transport equation, and read as: for a given force field ${\bf}:(0, T]\times \Omega\rightarrow \mathbb{R}^d$, find a velocity field $\bu:(0, T]\times \Omega\rightarrow \mathbb{R}^d$, and pressure and temperature fields $p,\, \theta:(0, T]\times \Omega\rightarrow \mathbb{R}$ such that $(\bu,\, p, \, \theta)$ satisfies the equations 
\begin{equation}
\begin{split}
\frac{\partial {\bu}}{\partial t} - \nu \Delta {\bu} +
\bm u \cdot  \nabla{\bu} + \nabla p &=
Ri \langle \bm 0,\theta \rangle + {\bf},\hspace{4mm}\text{in}\hspace{2mm}(0, T]\times \Omega,\\
{\nabla\cdot\bu} &= 0, \,\hspace{1mm}\quad\text{in}\hspace{2mm}(0, T]\times \Omega,\\
\frac{\partial {\theta}}{\partial t} - \kappa \Delta {\theta} +
(\bu\cdot\nabla){\theta} & = \Psi, \,\hspace{1mm}\quad\text{in}\hspace{2mm}(0, T]\times \Omega,
\end{split}
\label{Bouss}
\end{equation}
with appropriate boundary and initial conditions. The problem is posed on a bounded domain with Lipschitz continuous boundary. Here, $\nu:=Re^{-1}$ is the dimensionless kinematic viscosity, where $Re$ denotes the Reynolds number, $Ri:=Gr/Re^2$ is the Richardson number which accounts for the gravitational force and the thermal expansion of the fluid, and $\kappa:=1/(Pr Re)$ is thermal diffusivity coefficient. The Rayleigh number is defined by $Ra=Ri·Re^2·Pr$, and higher Ra leads to more complex physics as well as more difficulties in numerically solving the system. The modular grad-div stabilization method is given as follows:
\begin{algorithm}\label{algBouss} Let body forces $\bf , \Psi$, initial velocity $\bu_0$ and temperature $ \theta_0$, {and the stabilization parameters $\gamma\geq 0$, $\beta\geq 0$} be given. Set $\bu_h^0$, and  $\theta_h^{\, 0}$ to be $L^2$-orthogonal projection of $\,\,\bu_0$ into $\bX_h$, and $ \theta_{\,0}$ in $Y_h$, respectively. Select an end time $T$, and a time step $\Delta t>0$ such that $T/\Delta t = N$. Then find $\left(\bu_h^{n+1}, p_h^{n+1}, \theta_h^{\, n+1}\right)\in (\bX_h, Q_h, Y_h), \,$  $(n=0,1,2,..., N-1)$,  via the following :\ \\ \ \\ 
\textbf{Step 1:} Compute $\left(\widetilde{\bu}_h^{n+1}, p_h^{n+1}, \widetilde{\theta}_h^{n+1}\right)\in (\bX_h, Q_h, Y_h)$ such that for each $\left(\bv_h, q_h, \chi_h\right)\in (\bX_h, Q_h, Y_h)$
\begin{align}
\frac{1}{\Delta t}\big(\widetilde{\bu}_h^{n+1}- \bu_h^{n}, \bv_h\big)  + \nu\big(\nabla\widetilde{\bu}_h^{n+1}, \nabla \bv_h\big) + b\,(\bu_h^{n}, \, \widetilde{\bu}_h^{\,n+1}, \,\bv_h)  - (p_h^{n+1}, \nabla\cdot \bv_h)\nonumber\\
 = Ri \left(\langle 0,  \,\widetilde{\theta}_h^{\,n} \rangle, \,\bv_h\right) + ({{\bf}}^{n+1}, \bv_h),\label{bousDisc1}\\
( \nabla\cdot\widetilde{\bu}_h^{n+1},\, q_h)  = 0,\label{bousDisc2}\\
\frac{1}{\Delta t}\big(\widetilde{\theta}_h^{n+1} - \widetilde{\theta}_h^{n}, \chi_h\big) +  \kappa\left(\nabla \widetilde{\theta}_h^{n+1}, \nabla \chi_h\right)  + b^*\,(\bu_h^{n}, \, \widetilde{\theta}_h^{\,n+1}, \,\chi_h) = (\Psi^{n+1}, \,\chi_h)\label{bousDisc3}.
\end{align}
\textbf{Step 2:} Compute $\bu_h^{\,n+1}\in \bX_h$ such that for each $ \bvarphi_h\in \bX_h,$
\begin{align}
\left( {\bu}^{n+1}_h, \bvarphi_h\right) + (\beta + \gamma\Delta t)(\nabla\cdot\bu_h^{n+1}, \nabla\cdot\bvarphi_h)& = (\widetilde{\bu}_h^{n+1}, \bvarphi_h) + \beta (\nabla\cdot\bu_h^n , \nabla\cdot \bvarphi_h)\label{bousDisc4}.
\end{align}
\end{algorithm}

\begin{remark}
We emphasize here that modular grad-div stabilization step can be applied for any multistep time discretization. Numerical analysis for the BDF2 case can be found in \cite{F18}. 
\end{remark}

\section{Stability Analysis}
We now focus on the stability of Algorithm~\ref{algBouss}. Our stability analysis shows that approximate solutions of Algorithm~\ref{algBouss} are stable without any time step restriction. We first present a lemma which gives a relation between solutions of Step 1 and Step 2, and necessary for the stability result.
\begin{lemma}\label{prestab}
Let ${\bu}_h^{n+1}$ be solutions to \eqref{bousDisc4}. Then it holds
\begin{align}
\|\widetilde{\bu}_h^{n+1}\|^2_{L^2}  = \|{\bu}_h^{n+1}\|^2_{L^2} 
	& + \|\widetilde{\bu}_h^{n+1}-{\bu}_h^{n+1}\|^2_{L^2} 
	 + 2\gamma\Delta t \|\nabla \cdot{\bu}_h^{n+1}\|^2_{L^2} \nonumber\\
	& \, \, + \beta \left(\|\nabla\cdot\bu_h^{n+1}\|^2_{L^2} - \|\nabla\cdot\bu_h^{n}\|^2_{L^2} + \|\nabla\cdot\left(\bu_h^{n+1}-\bu_h^n\right)\|^2_{L^2} \right).
\end{align}
\end{lemma}
\begin{proof}
Set ${\bvarphi_h} = {\bu}_h^{n+1}$ in \eqref{bousDisc4} which yields
\begin{align}
(\widetilde{\bu}_h^{n+1}, {\bu}^{n+1}_h) 
	& = \| {\bu}^{n+1}_h\|^2_{L^2}
	+ (\beta + \gamma\Delta t)\|\nabla\cdot\bu_h^{n+1}\|^2_{L^2}
	- \beta (\nabla\cdot\bu_h^n , \nabla\cdot\bu_h^{n+1})\label{stabBousvel1}.
\end{align}
Apply the polarization identity on the left hand side and on the last right hand side terms to get:
\begin{align*}
(\widetilde{\bu}_h^{n+1}, {\bu}^{n+1}_h)
	& = \frac{1}{2}\bigg(\|\widetilde{\bu}_h^{n+1}\|^2_{L^2} + \|\bu_h^{n+1}\|^2_{L^2} - \|\widetilde{\bu}_h^{n+1} - \bu_h^{n+1}\|^2_{L^2}\bigg),\\
- \beta (\nabla\cdot\bu_h^n , \nabla\cdot\bu_h^{n+1})
	& = -\frac{\beta}{2}\bigg(
\|\nabla\cdot{\bu}_h^{n}\|^2_{L^2} + \|\nabla\cdot\bu_h^{n+1}\|^2_{L^2} - \|\nabla\cdot\left({\bu}_h^{n} - \bu_h^{n+1}\right)\|^2_{L^2}\bigg).
\end{align*} 
Inserting these estimates into \eqref{stabBousvel1}, rearranging terms and multiplying by $2$ gives the desired estimates.
\end{proof}
We now present the main stability result. 
\begin{lemma}
Assume that $\bf \in L^2(0,T; \bH^{-1}(\Omega))$ and $\Psi\in L^2(0,T; H^{-1}(\Omega))$. Then solutions to Algorithm~\ref{algBouss} satisfy the following: for any $\Delta t>0$
\begin{gather}
\|\bu_h^{N}\|^2_{L^2} 
	+ \beta \|\nabla\cdot\bu_h^{N}\|^2_{L^2} 
	+\sum_{n=0}^{N-1}\bigg(\|\widetilde{\bu}_h^{n+1} - \bu_h^{n+1}\|^2_{L^2} +  \|\widetilde{\bu}_h^{n+1} - \bu_h^{n}\|^2_{L^2} \bigg)
	+ \beta \sum_{n=0}^{N-1}\|\nabla\cdot\left({\bu}_h^{n+1} - \bu_h^{n}\right)\|^2_{L^2}\nonumber\\ 
	+ 2\gamma\Delta t\sum_{n=0}^{N-1}\|\nabla\cdot \bu_h^{n+1}\|^2_{L^2} \
	 + \nu\,\Delta t\sum_{n=0}^{N-1}\|\nabla\widetilde{\bu}_h^{n+1}\|^2_{L^2}
	 \leq 2 C_P^2 Ri^2\nu^{-1} T M + 2 \nu^{-1} \Delta t \sum_{n=0}^{N-1}\|\bf^{n+1}\|_{-1}^2,
\end{gather}
and
\begin{gather}
\|\widetilde{\theta}_h^{N}\|^2_{L^2}  
	+\sum_{n=0}^{N-1}\|\widetilde{\theta}_h^{n+1} - \widetilde{\theta}_h^{n}\|^2_{L^2} 
	 + \kappa \,\Delta t\sum_{n=0}^{N-1}\|\nabla\widetilde{\theta}_h^{n+1}\|^2_{L^2} 
	 \leq M,
\end{gather}
where $M:=\left(\|\widetilde{\theta}_h^{0}\|_{L^2}^2 +\kappa^{-1}\Delta t\sum\limits_{n=0}^{N-1}\|\Psi_{h}^{n+1}\|_{-1}^2\right).$
\end{lemma}
\begin{proof}
We first prove the temperature stability result. Set $\chi_h=2\Delta t\widetilde{\theta}_h^{n+1}$ in \eqref{bousDisc3}, which vanishes the non-linear term and leaves:
\begin{align*}
\bigg(\|\widetilde{\theta}_h^{n+1}\|^2_{L^2} - \|\widetilde{\theta}_h^{n}\|^2_{L^2} + \|\widetilde{\theta}_h^{n+1}- \widetilde{\theta}_h^{n}\|^2_{L^2}\bigg) 
	+ 2 \,\kappa\, \Delta t\|\nabla\widetilde{\theta}_h^{n+1}\|^2_{L^2}
	 = 2 \, \Delta t (\Psi^{n+1}, \widetilde{\theta}_h^{n+1}).
\end{align*}
Apply the Cauchy-Schwarz and Young's inequalities on the right hand side term to get
\begin{align*}
2 \, \Delta t\, (\Psi^{n+1}, \widetilde{\theta}_h^{n+1})&\leq {\kappa^{-1}}{\Delta t }\|\Psi^{n+1}\|_{-1}^2 + \kappa\, {\Delta t \,}\|\nabla\widetilde{\theta}_h^{n+1}\|^2_{L^2}.
\end{align*}
Inserting this estimate produces
\begin{align}
\bigg(\|\widetilde{\theta}_h^{n+1}\|^2_{L^2} - \|\widetilde{\theta}_h^{n}\|^2_{L^2} + \|\widetilde{\theta}_h^{n+1}- \widetilde{\theta}_h^{n}\|^2_{L^2}\bigg) 
	+  \,\kappa\, \Delta t\|\nabla\widetilde{\theta}_h^{n+1}\|^2_{L^2}
	 \leq \, \kappa^{-1} \Delta t \|\Psi^{n+1}\|_{-1}^2\label{stabforvel}.
\end{align}
Dropping the non-negative third left hand side term and summing over time steps gives the stability bound for the temperature. For the stability of the velocity, set $(\bv_h, q_h)= (2\, \Delta t\,\widetilde{\bu}_h^{n+1}, p_h^{n+1})$ in \eqref{bousDisc1}-\eqref{bousDisc2} to get 
\begin{multline*}
\bigg(\|\widetilde{\bu}_h^{n+1}\|^2_{L^2} - \|{\bu}_h^{n}\|^2_{L^2} + \|\widetilde{\bu}_h^{n+1}- {\bu}_h^{n}\|^2_{L^2}\bigg) 
	+ 2 \,\nu\, \Delta t\|\nabla\widetilde{\bu}_h^{n+1}\|^2_{L^2}
	 = 2\,Ri \Delta t \left(\langle 0,  \,\widetilde{\theta}_h^{\,n} \rangle, \,\bv_h\right) +  2 \, \Delta t (\bf^{n+1}, \widetilde{\bu}_h^{n+1}).
\end{multline*}
Apply the Cauchy-Schwarz and Young's inequalities on the right hand side terms to produce
\begin{align*}
2 \, \Delta t\, (\bf^{n+1}, \widetilde{\bu}_h^{n+1})&\leq {2\,\nu^{-1} \Delta t }\|\bf^{n+1}\|_{-1}^2 + \frac{\nu\, \Delta t \,}{2}\|\nabla\widetilde{\bu}_h^{n+1}\|^2_{L^2},\\
2\,Ri\Delta t \left(\langle 0,  \,\widetilde{\theta}_h^{\,n} \rangle, \,\bv_h\right)
&\leq {2\, C_P^2\, Ri^2\nu^{-1}\Delta t}\|\widetilde{\theta}^{n}_h\|^2_{{L^2}} + \frac{\nu\,\Delta t}{2}\,\|\nabla\widetilde{\bu}_h^{n+1}\|^2_{L^2}.
\end{align*}
Insert these estimates and rearrange terms to obtain
\begin{gather*}
\bigg(\|\widetilde{\bu}_h^{n+1}\|^2_{L^2} - \|{\bu}_h^{n}\|^2_{L^2} + \|\widetilde{\bu}_h^{n+1}- {\bu}_h^{n}\|^2_{L^2}\bigg) 
	+  \,\nu\, \Delta t\|\nabla\widetilde{\bu}_h^{n+1}\|^2_{L^2}
	 \leq {2\, C_P^2\, Ri^2\,\nu^{-1}\Delta t}\|\widetilde{\theta}^{n}_h\|^2_{{L^2}} \, + \, \frac{\Delta t }{\nu}\|\bf^{n+1}\|_{-1}^2.
\end{gather*}
Now use Lemma~\ref{prestab} on the left hand side to obtain
\begin{gather*}
\bigg(\|{\bu}_h^{n+1}\|^2_{L^2} - \|{\bu}_h^{n}\|^2_{L^2} \bigg)	
	+\beta \bigg(\|\nabla\cdot\bu_h^{n+1}\|^2_{L^2} - \|\nabla\cdot\bu_h^{n}\|^2_{L^2}\bigg)
	+  \bigg( \|\widetilde{\bu}_h^{n+1}- {\bu}_h^{n+1}\|^2_{L^2} + \|\widetilde{\bu}_h^{n+1}- {\bu}_h^{n}\|^2_{L^2}\bigg) \\
	+ \beta\|\nabla\cdot\left(\bu_h^{n+1}-\bu_h^n\right)\|^2_{L^2}
	 + 2\,\gamma\,\Delta t\, \|\nabla \cdot{\bu}_h^{n+1}\|^2_{L^2}
	 +  \nu \Delta t\,\|\nabla\widetilde{\bu}_h^{n+1}\|^2_{L^2}
	 \leq 
	 {2\, C_P^2\, Ri^2\,\nu^{-1}\Delta t}\|\widetilde{\theta}^{n}_h\|^2_{{L^2}}
	\, + \, {\nu^{-1}}{\Delta t }\|\bf^{n+1}\|_{-1}^2. 
\end{gather*}
Notice that from \eqref{stabforvel}, one can get
\begin{align*}
\Delta t\sum_{n=0}^{N-1}\|\widetilde{\theta}_h^{n}\|_{L^2}^2\leq \Delta t \,N\left(\|\widetilde{\theta}_h^{0}\|_{L^2}^2 +\kappa^{-1}\Delta t\sum_{n=0}^{N-1}\|\Psi_{h}^{n+1}\|_{-1}^2\right)=:T\,M.
\end{align*}
Summing over time steps with this estimate and rearranging terms finishes the proof.
\end{proof}
\section{Error Analysis}
In this section, we show that solutions of the proposed algorithm converge to the true solutions of \eqref{Bouss}. We denote true Boussinesq solutions at time level $t^{n+1}, $ by $$\bu^{n+1}:= \bu(t^{n+1}),\,\, p^{n+1}:=p(t^{n+1}), \,\, \theta^{n+1}:=\theta(t^{n+1}), \hspace{2mm}\, \, \, n=-1,0,1,..., N-1.$$
The error analysis needs the following error decompositions at time level $t^{n+1}$: 
\begin{align*}
\be_{\widetilde{u}}^{n+1}:&=\bu^{n+1}- \widetilde{\bu}_h^{n+1} =\left(\bu^{n+1}- {\bm P}_{\bV_h}({\bu}^{n+1}) \right) - \left(\widetilde{\bu}_h^{n+1} - {\bm P}_{\bV_h}({\bu}^{n+1} )\right)=:\bfeta_{\bu}^{n+1}- \bLambda_{\bu,h}^{n+1},\\
\be_{{u}}^{n+1}:& =\bu^{n+1}- {\bu}_h^{n+1} =\left(\bu^{n+1} - {\bm P}_{\bV_h}({\bu}^{n+1}) \right) - \left({\bu}_h^{n+1} - {\bm P}_{\bV_h}({\bu}^{n+1} \right)=:\bfeta_{\bu}^{n+1} - \bphi_{\bu,h}^{n+1},\\
e_{\widetilde{\theta}}^{n+1}:& =\theta^{n+1}- \widetilde{\theta}_h^{n+1} =\left(\theta^{n+1} - P_{{Y}_h}({\theta}^{n+1}) \right) - \left(\widetilde{\theta}_h^{n+1} - P_{{Y}_h}({\theta}^{n+1}) \right)=:\eta_{\theta}^{n+1} - \Lambda_{\theta,h}^{n+1},
\end{align*}
where {$\bm P_{\bV_h}({\bu}^{n+1} )$ is the $L^2$-best approximation of $\bu^{n+1}$ in $\bV_h$, and $P_{Y_h}({\theta}^{n+1} )$ the $L^2$-best approximation of $\theta^{n+1}$ in $Y_h$.} Moreover, $\bfeta_{\bu}^{n+1}, \, \eta_{\theta}^{n+1}$ are interpolation errors, and $\bLambda_{\bu,h}^{n+1}, \bphi_{\bu,h}^{n+1}\in \bX_h$ and $\, \Lambda_{\theta,h}^{n+1}\in Y_h$ are finite element errors. We now present {the} following result which helps us to prove the convergence theorem.
\begin{lemma}\label{yarlemmaBouss}
Consider the second step of Algorithm~\ref{algBouss}. Then it holds:
\begin{align*}
\|\bLambda_{\bu, h}^{n+1}\|^2_{L^2}
	\, & \geq \, \|\bphi_{\bu,h}^{n+1}\|^2_{L^2}
	 \, + \, \|\bLambda_{\bu, h}^{n+1} - \bphi_{\bu,h}^{n+1}\|^2_{L^2}
	 \, + \, \beta\left( \|\nabla\cdot\bphi_{\bu,h}^{n+1}\|^2_{L^2} - \|\nabla\cdot\bphi_{\bu,h}^{n}\|^2_{L^2}\right)\nonumber\\
	 \, & + \, \frac{\beta}{2}\|\nabla\cdot(\bphi_{\bu,h}^{n+1}-\bphi_{\bu,h}^{n})\|^2_{L^2}
	 \,  + \, \gamma\Delta t \|\nabla\cdot \bphi_{\bu,h}^{n+1}\|^2_{L^2} 
	 \, - \, \beta\Delta t \|\nabla\cdot \bphi_{\bu,h}^{n}\|^2_{L^2}\nonumber\\ 
	 \, & - \,  \beta\,(1 + 2\,\Delta t)\|\nabla\bfeta_{\bu, t}\|^2_{L^{2}(t^n, t^{n+1}; L^2(\Omega))} 
	 \, - \, \gamma\,\Delta t \|\nabla\bfeta_{\bu}^{n+1}\|^2_{L^2}. 
\end{align*}
\end{lemma}
\begin{proof}
The true velocity solution at time level $t^{n+1}$ satisfies the following:
\begin{align*}
\left( {\bu}^{n+1}, \bv_h\right) + (\beta + \gamma\Delta t)(\nabla\cdot\bu^{n+1}, \nabla\cdot\bv_h)& = ({\bu}^{n+1}, \bv_h) + \beta (\nabla\cdot\bu^n , \nabla\cdot \bv_h).
\end{align*}
Subtract this system from the second step of Algorithm~\ref{algBouss}. Then using error {notation} and rearranging terms produces
\begin{align*}
\left( {\be}_{\bu}^{n+1}, \bv_h\right)
	 + \beta (\nabla\cdot\left({\be}_{\bu}^{n+1} - {\be}_{\bu}^{n} \right), \nabla\cdot\bv_h)
	+ \gamma\Delta t(\nabla\cdot{\be}_{\bu}^{n+1}, \nabla\cdot\bv_h)	
	& = ({\be}_{\widetilde{\bu}}^{n+1}, \bv_h).
\end{align*}
Using error decomposition and setting $\bv_h = \bphi_{\bu,h}^{n+1}$  yields 
\begin{gather*}
\|\bphi_{\bu,h}^{n+1}\|^2_{L^2} 
	 \, + \, \frac{\beta}{2}( \|\nabla\cdot\bphi_{\bu,h}^{n+1}\|^2_{L^2} - \|\nabla\cdot\bphi_{\bu,h}^{n}\|^2_{L^2} +  \|\nabla\cdot(\bphi_{\bu,h}^{n+1}-\bphi_{\bu,h}^{n})\|^2_{L^2})
	 \,  \, + \, \gamma\Delta t \|\nabla\cdot \bphi_{\bu,h}^{n+1}\|^2_{L^2} \nonumber\\
	 \, = \, (\bfeta_{\bu}^{n+1}, \,\,\bphi_{\bu,h}^{n+1})
	 \, + \,{\beta}(\nabla\cdot(\bfeta_{\bu}^{n+1} - \bfeta_{\bu}^{n}),\,\,\nabla\cdot\bphi_{\bu,h}^{n+1})
	 \, + \,{\gamma\,\Delta t}(\nabla\cdot\bfeta_{\bu}^{n+1},\,\,\nabla\cdot\bphi_{\bu,h}^{n+1})
	 \, - \, (\bfeta_{\bu}^{n+1} -\bLambda_{\bu,h}^{n+1}, \,\,\bphi_{\bu,h}^{n+1})\nonumber.
\end{gather*}
Now add $\mp {\beta}(\nabla\cdot(\bfeta_{\bu}^{n+1} - \bfeta_{\bu}^{n}),\,\,\nabla\cdot\bphi_{\bu,h}^{n})$ and notice that
$(\bfeta_{\bu}^{n+1} , \,\,\bphi_{\bu,h}^{n+1})= 0 .$
This produces
\begin{gather*}
\|\bphi_{\bu,h}^{n+1}\|^2_{L^2} 
	 \, + \, \frac{\beta}{2}\left( \|\nabla\cdot\bphi_{\bu,h}^{n+1}\|^2_{L^2} - \|\nabla\cdot\bphi_{\bu,h}^{n}\|^2_{L^2} +  \|\nabla\cdot(\bphi_{\bu,h}^{n+1}-\bphi_{\bu,h}^{n})\|^2_{L^2}\right)
	   \, + \, \gamma\Delta t \|\nabla\cdot \bphi_{\bu,h}^{n+1}\|^2_{L^2} \nonumber\\
	 \, =  \,{\beta}(\nabla\cdot(\bfeta_{\bu}^{n+1} - \bfeta_{\bu}^{n}),\,\,\nabla\cdot(\bphi_{\bu,h}^{n+1}-\bphi_{\bu,h}^{n} ))
	  \, + \ {\beta}(\nabla\cdot(\bfeta_{\bu}^{n+1} - \bfeta_{\bu}^{n}),\,\,\nabla\cdot\bphi_{\bu,h}^{n})\nonumber\\
	 \, + \,{\gamma\,\Delta t}(\nabla\cdot\bfeta_{\bu}^{n+1},\,\,\nabla\cdot\bphi_{\bu,h}^{n+1})
	 \, + \, (\bLambda_{\bu,h}^{n+1}, \,\,\bphi_{\bu,h}^{n+1})\nonumber.
\end{gather*}
To bound the first three right hand side terms, apply the Cauchy-Schwarz, and the Young's inequalities to get
\begin{align*}
{\beta}(\nabla\cdot(\bfeta_{\bu}^{n+1} - \bfeta_{\bu}^{n}),\,\,\nabla\cdot(\bphi_{\bu,h}^{n+1}-\bphi_{\bu,h}^{n}) 
	\,& \leq \, \beta \,  \|\nabla\left(\bfeta_{\bu}^{n+1} - \bfeta_{\bu}^{n}\right)\|_{L^2}\|\nabla\cdot(\bphi_{\bu,h}^{n+1}-\bphi_{\bu,h}^{n})\|_{L^2}\\
	\,& \leq \,  \beta \, \Delta t\|\nabla\bfeta_{\bu,t}\|^2_{L^2(t^n, t^{n+1};L^2(\Omega))}
	\, + \, \frac{\beta}{4}\|\nabla\cdot(\bphi_{\bu,h}^{n+1}-\bphi_{\bu,h}^{n})\|^2_{L^2},\\
{\beta}(\nabla\cdot(\bfeta_{\bu}^{n+1} - \bfeta_{\bu}^{n}),\,\,\nabla\cdot\bphi_{\bu,h}^{n})
	\,& \leq \,  \beta \, \|\nabla(\bfeta_{\bu}^{n+1} - \bfeta_{\bu}^{n})\|_{L^2}\|\nabla\cdot\bphi_{\bu,h}^{n}\|_{L^2}\\
\, & \leq \, \frac{ \beta \,  }{2}\|\nabla\bfeta_{\bu,t}\|^2_{L^2(t^n, t^{n+1};L^2(\Omega))}
	\, + \, \frac{\beta\,\Delta t}{2}\|\nabla\cdot\bphi_{\bu,h}^{n}\|^2_{L^2},\\
{\gamma\,\Delta t}(\nabla\cdot\bfeta_{\bu}^{n+1},\,\,\nabla\cdot\bphi_{\bu,h}^{n+1})
	 \, \leq \, {\gamma\,\Delta t}&\|\nabla\bfeta_{\bu}^{n+1}\|_{L^2} \|\nabla\cdot\bphi_{\bu,h}^{n+1} \|_{L^2}
	 \, \leq \, \frac{\gamma\,\Delta t}{2}\|\nabla\bfeta_{\bu}^{n+1} \|^2_{L^2} 
	 \, + \, \frac{\gamma\,\Delta t}{2}\|\nabla\cdot\bphi_{\bu,h}^{n+1}\|^2_{L^2}.
\end{align*}
For the last term, use the polarization identity to obtain
\begin{align*}
(\bLambda_{\bu,h}^{n+1}, \,\,\bphi_{\bu,h}^{n+1})\, = \, \frac{1}{2}\big(\|\bLambda_{\bu,h}^{n+1}\|^2_{L^2}
	 \, + \, \|\bphi_{\bu,h}^{n+1}\|^2_{L^2} 
	 \, - \,  \|\bLambda_{\bu,h}^{n+1}-\bphi_{\bu,h}^{n+1}\|^2_{L^2}\,\big).
\end{align*}
Plugging these estimates into velocity error equation, and reducing yields
\begin{gather*}
\frac{1}{2}\|\bphi_{\bu,h}^{n+1}\|^2_{L^2} 
	 \, + \, \frac{\beta}{2}\left( \|\nabla\cdot\bphi_{\bu,h}^{n+1}\|^2_{L^2} - \|\nabla\cdot\bphi_{\bu,h}^{n}\|^2_{L^2}\right)
	  \, + \,\frac{\beta}{4} \|\nabla\cdot(\bphi_{\bu,h}^{n+1}-\bphi_{\bu,h}^{n})\|^2_{L^2}
	 \, + \, \frac{\gamma\Delta t}{2} \|\nabla\cdot \bphi_{\bu,h}^{n+1}\|^2_{L^2} \nonumber\\
	 \, \leq \, \frac{ \beta \, { \, \left(1 + 2\Delta t\right)}\,}{2}\|\nabla\bfeta_{\bu,t}\|^2_{L^2(t^n, t^{n+1};L^2(\Omega))}
	\, + \, \frac{\beta\Delta t}{2}\|\nabla\cdot\bphi_{\bu,h}^{n}\|^2_{L^2}
	\, + \,\frac{\gamma\,\Delta t}{2}\|\nabla\bfeta_{\bu}^{n+1} \|^2_{L^2}\nonumber\\
	\, + \,\frac{1}{2}\big(\|\bLambda_{\bu,h}^{n+1}\|^2_{L^2} 		 
	 \, - \, \|\bLambda_{\bu,h}^{n+1}-\bphi_{\bu,h}^{n+1}\|^2_{L^2}\,\big). 
\end{gather*}
Multiplying by $2\Delta t$, and rearranging terms gives the desired estimate.
\end{proof}
We now prove an error estimate to Algorithm~\ref{algBouss}.
\begin{theorem}\label{convBouss} Assume that true solution $(\bu, \theta, p)$ satisfies the regularity conditions :
\begin{gather*}
\bu \in L^{\infty}(0,T; \bm H^{k+1}(\Omega) \textcolor{red}{\,\cap\,\bm H^{3}(\Omega))}, \hspace{1mm} \bu_t  \in L^{\infty}(0,T; \bm H^{k+1}(\Omega)),\hspace{1mm}
\\
\bu_{tt} \in L^{2}(0,T; \bm L^2(\Omega)),  \hspace{1mm}p \in L^{\infty}(0,T;L^2(\Omega)),\hspace{1mm}\\
\theta\in L^{\infty}(0,T;H^{k+1}(\Omega) \textcolor{red}{\, \cap \, H^{3}(\Omega))}, \hspace{1mm} \theta_t \in L^{\infty}(0,T;H^{k+1}(\Omega)),\hspace{2mm}
\theta_{tt} \in L^{2}(0,T;  L^2(\Omega)).
\label{regularityBouss}
\end{gather*}
Let $\left(\widetilde{\bu}^{n+1}_h, p_h^{n+1},\widetilde{\theta}^{n+1}_h, \bu^{n+1}_h\right)$ be solution to Algorithm~\ref{algBouss}, and $(\bX_h, Q_h, Y_h)$ is given by $(\bP_k, P_{k-1}, P_{k})$. Then {the} errors satisf{y} the bound \vspace{1mm}
\begin{gather}
\|\be_{\bu}^{N}\|^2_{L^2} 
	\, + \,  \|\be_{\widetilde{\theta}}^{N}\|^2_{L^2}
	\, + \, \beta \,\|\nabla\cdot \be_{\bu}^{N}\|^2_{L^2} 
	\, + \, \gamma\||\nabla\cdot \be_{\bu}|\|^2_{2,0}
	\, + \,\nu \||\nabla \be_{\widetilde{\bu}}|\|^2_{2,0}
	\, + \,	\kappa\, \||\nabla \be_{\widetilde{\theta}}|\|^2_{2,0}\nonumber\\
	\vspace{1cm}
	\, \leq \,C \left(h^{2k+2} + \Delta t h^{2k-1} + h^{2k} + \Delta t^2\right),
	\end{gather}
where $C$ is a generic constant independent of the time step and mesh size.
\end{theorem}
\begin{proof}
The proof is divided into four steps since it is very long and technical. In the first step, the error equations are obtained by splitting the velocity and magnetic errors into approximation errors and finite element remainders. In the second step, all right hand side terms of the error equations are bounded below. The third step applies the discrete Gronwall lemma, and the last step the triangle inequality for the error terms. \, \vspace{2mm}\\
\textbf{Step 1:} \textbf{[}\textbf{\textit{The derivation of error equations}}.\textbf{]}\, \vspace{3mm}\\
 True solution $(\bu, p, \theta)$ satisfies the equations, \textcolor{red}{$\forall \, \bv_h \in \bV_h$ and $\forall \, \chi_h\in Y_h$,}
\begin{align}
	\left( \frac{{\bu}^{n+1} -\bu^n}{\Delta t}, \,\,\bv_h\right)
	 \, + \, \nu (\nabla{\bu}^{n+1}, \,\,\nabla\bv_h )
	 \, + \,b\left({\bu}^{n} , \,\,{\bu}^{n+1} ,\,\,\bv_h\right) 
	 \, - \,(p^{n+1}, \,\,\nabla\cdot\bv_h)
	  \, = Ri \left(\langle 0,  \,\theta^{n} \rangle, \,\bv_h\right)\nonumber\\
	   \, + \, (\bf^{n+1}, \,\,\bv_h)
	  \, - \, E_1(\bu, \theta, \bv_h),\label{trueBous1}\\
	 \left( \frac{{\theta}^{n+1} -\theta^n}{\Delta t}, \,\chi_h\right)
	  \, + \,\kappa (\nabla{\theta}^{n+1}, \,\nabla\chi_h )
	  \, +  \, b^*\left(\bu^{n}, \,{\theta}^{n+1}, \,\chi_h\right)
	  \, = \, (\Psi^{n+1}, \,\chi_h)
	  \, - \, E_2(\bu, \theta, \chi_h)	   
	   \label{trueBous3},
\end{align}
where $E_{1}(\bu, \theta, \bv_h)$ and $E_{2}(\bu, \theta, \chi_h)$ are consistency errors and given by
\begin{align*}
E_{1}(\bu, \theta, \bv_h)
	:&= \left(\bu_{t}^{n+1}- \frac{{\bu}^{n+1} -\bu^n}{\Delta t} , \,\,\bv_h\right)
	\, + \, b(\bu^{n+1}- \bu^{n}, \bu^{n+1}, \bv_h)
	\, - \, Ri \left(\langle 0, \,\theta^{n+1} - \theta^{n} \rangle, \,\bv_h\right)\\
E_{2}(\bu, \theta, \chi_h)
	:&= \left(\theta_{t}^{n+1}- \frac{{\theta}^{n+1} -\theta^n}{\Delta t} , \,\,\chi_h\right)
	\ + \ b^*(\bu^{n+1}- \bu^{n}, \theta^{n+1}, \chi_h).	
 \end{align*}
Subtract the first step of Algorithm~\ref{algBouss} from \eqref{trueBous1}-\eqref{trueBous3}, and use error {notation} to produce: \textcolor{red}{for any $q_h\in Q_h$}
\begin{multline}
\left(\frac{\be_{\widetilde{\bu}}^{n+1}- \be_{\bu}^{n}}{\Delta t},\,\,\bv_h\right)
	\, + \, \nu\left(\nabla\be_{\widetilde{\bu}}^{n+1},  \,\,\nabla\bv_h\right)
	\, + \, b(\bu^{n}, \bu^{n+1}, \bv_h)
	\, - \, b(\bu^{n}_h, \widetilde{\bu}^{n+1}_h, \bv_h) 
	\, - \, \left(p^{n+1}- q_h,\,\,\nabla\cdot\bv_h\right)\\
	\, - \, Ri(\langle \bm 0, e_{\widetilde{\theta}}^n \rangle, \bv_h)
	\, + \,E_1(\bu, \theta, \bv_h) = 0\label{errvel1},
\end{multline}
and
\begin{multline}
\left(\frac{e_{\widetilde{\theta}}^{n+1}- e_{\widetilde{\theta}}^{n}}{\Delta t},\,\,\chi_h\right)
	\, + \, \kappa\left(\nabla e_{\widetilde{\theta}}^{n+1},  \,\,\nabla\chi_h\right)
	\, + \, b^*(\bu^{n}, \theta^{n+1}, \chi_h)
	\, - \, b^*(\bu^{n}_h, \widetilde{\theta}^{n+1}_h, \chi_h)
	\, + \,E_2(\bu, \theta, \chi_h) = 0.\label{errtemp1}
\end{multline}
Using error decomposition and setting $\bv_h = 2\Delta t \bLambda_{\bu,h}^{n+1}$ in \eqref{errvel1}, and $\chi_h= 2\Delta t \Lambda_{{{\theta}},h}^{n+1}$ in \eqref{errtemp1} yields: for any $q_h\in Q_h,$
\begin{multline}
\|\bLambda_{\bu,h}^{n+1}\|^2_{L^2} - \|\bphi_{\bu,h}^{n}\|^2_{L^2}
	\, + \, \|\bLambda_{\bu,h}^{n+1} - \bphi_{\bu,h}^{n} \|^2_{L^2}
	\, + \, 2\nu\Delta t \|\nabla\bLambda_{\bu,h}^{n+1}\|^2_{L^2}\\
	\, = \,
	2\left(\bfeta_{\bu}^{n+1}-\bfeta_{\bu}^{n},\,\, \bLambda_{\bu,h}^{n+1}\right)
	\, + \, 2\nu\Delta t \left(\nabla\bfeta_{\bu}^{n+1},\,\, \nabla\bLambda_{\bu,h}^{n+1}\right)
	\, + \, 2\Delta t\bigg( 
	 b(\bu^{n}, \bu^{n+1}, \bLambda_{\bu,h}^{n+1})
	\, - \, b(\bu^{n}_h, \widetilde{\bu}^{n+1}_h, \bLambda_{\bu,h}^{n+1})\bigg)\\
	\, -2\, Ri\,\Delta t(\langle \bm 0, e_{\widetilde{\theta}}^n \rangle, \bLambda_{\bu,h}^{n+1})
	\, - \, 2\,\Delta t\, \left(p^{n+1}- q_h,\,\,\nabla\cdot\bLambda_{\bu,h}^{n+1}\right)
	 \, + \, 2\,\Delta t E_{1}(\bu, \theta, \bLambda_{\bu,h}^{n+1})\label{errvel2},
\end{multline}
and
\begin{multline}
\|\Lambda_{{\theta},h}^{n+1}\|^2_{L^2} - \|\Lambda_{{\theta},h}^{n}\|^2_{L^2}
	\, + \, \|\Lambda_{{\theta},h}^{n+1} - \Lambda_{{\theta},h}^{n} \|^2_{L^2}
	\, + \, 2\,\kappa \,\Delta t \|\nabla\Lambda_{{\theta},h}^{n+1}\|^2_{L^2}\\
	\, = \,
	2\,\left(\eta_{{\theta}}^{n+1}-\eta_{{\theta}}^{n},\,\, \Lambda_{{\theta},h}^{n+1}\right)
	\, + \, 2\,\kappa\,\Delta t \,\left(\nabla\eta_{{\theta}}^{n+1},\,\, \nabla\Lambda_{{\theta},h}^{n+1}\right)
	\, + \, 2\,\Delta t\,\bigg( 
	 b^*(\bu^{n}, \theta^{n+1}, \, \Lambda_{{\theta},h}^{n+1})
	\, - \, b^*(\bu^{n}_h, \widetilde{\theta}^{n+1}_h,\, \Lambda_{{\theta},h}^{n+1})\bigg)\\
	+ \, 2\,\Delta t E_{2}(\bu, \theta, \Lambda_{{\theta},h}^{n+1})\label{errtemp2}.
 \end{multline}
Now, we rewrite the nonlinear terms in Equation \ref{errvel2} by adding and subtracting terms as follows:
\textcolor{red}{
 \begin{align*}  
	 b(\bu^{n}, \bu^{n+1}, \bLambda_{\bu,h}^{n+1})
	\, &- \,b(\bu^{n}_h, {\bu}^{n+1}, \bLambda_{\bu,h}^{n+1})
	\, + \, b(\bu^{n}_h, {\bu}^{n+1}, \bLambda_{\bu,h}^{n+1})
	\, - \, b(\bu^{n}_h, \widetilde{\bu}^{n+1}_h, \bLambda_{\bu,h}^{n+1})\\
	\,& = \,b(\be_{\bu}^{n}, \bu^{n+1}, \bLambda_{\bu,h}^{n+1})
	\, + \,b(\bu^{n}_h, \be_{\widetilde{\bu}}^{n+1}, \bLambda_{\bu,h}^{n+1})\\
	\, &= \,b(\be_{\bu}^{n}, \bu^{n+1}, \bLambda_{\bu,h}^{n+1})
	\, + \,b(\bu^{n}_h, \bfeta_{\bu}^{n+1}, \bLambda_{\bu,h}^{n+1}),\\
	\, &= \,b(\be_{\bu}^{n}, \bu^{n+1}, \bLambda_{\bu,h}^{n+1})
	\, + \,b(\bu^{n}_h, \bfeta_{\bu}^{n+1}, \bLambda_{\bu,h}^{n+1}) 
	- \, b({\bu}^{n}, \bfeta_{\bu}^{n+1}, \bLambda_{\bu,h}^{n+1})
	 \, + \, b({\bu}^{n}, \bfeta_{\bu}^{n+1}, \bLambda_{\bu,h}^{n+1})\\
	 \, &= \,\,b(\be_{\bu}^{n}, \bu^{n+1}, \,\bLambda_{\bu,h}^{n+1}) 
	 \, - \, b(\be_{\bu}^{n}, \,\, \bfeta_{\bu}^{n+1}, \,\bLambda_{\bu,h}^{n+1})
	 \, + \, b({\bu}^{n}, \bfeta_{\bu}^{n+1}, \bLambda_{\bu,h}^{n+1}).
\end{align*}
}
Plugging this rearrangement into Equation \ref{errvel2}, we have 
\begin{align}
	\|\bLambda_{\bu,h}^{n+1}\|^2_{L^2} 
	& - \|\bphi_{\bu,h}^{n}\|^2_{L^2}
	\,  + \, \|\bLambda_{\bu,h}^{n+1} - \bphi_{\bu,h}^{n} \|^2_{L^2}
	\, + \, 2\nu\Delta t \|\nabla\bLambda_{\bu,h}^{n+1}\|^2_{L^2}\nonumber\\
	\,& = \,
	2\left(\bfeta_{\bu}^{n+1}-\bfeta_{\bu}^{n},\,\, \bLambda_{\bu,h}^{n+1}\right)
	\, + \, 2\nu\Delta t \left(\nabla\bfeta_{\bu}^{n+1},\,\, \nabla\bLambda_{\bu,h}^{n+1}\right)\nonumber \\
	& \, \,\,\,\,\, \,\,\,\ \textcolor{red}{+ \, 2\,  \Delta \left(\, b(\be_{\bu}^{n}, \bu^{n+1}, \,\bLambda_{\bu,h}^{n+1}) 
	 \, - \, b(\be_{\bu}^{n}, \,\, \bfeta_{\bu}^{n+1}, \,\bLambda_{\bu,h}^{n+1})
	 \, + \, b({\bu}^{n}, \bfeta_{\bu}^{n+1}, \bLambda_{\bu,h}^{n+1})\right) }\nonumber\\
	 \, & \, \,\,\,\,\, \,\,\,\ - \, 2 Ri\Delta t(\langle \bm 0, e_{\widetilde{\theta}}^n \rangle ,  \bLambda_{\bu,h}^{n+1}) 
	\, - \, 2\,\Delta t\, \left(p^{n+1}- q_h,\,\,\nabla\cdot\bLambda_{\bu,h}^{n+1}\right)
	 \, + \, 2\,\Delta t E_{1}(\bu, \,\, \theta, \,\, \bLambda_{\bu,h}^{n+1}).\label{errvel3}
\end{align}
Using similar treatment for the nonlinear terms in Equation \ref{errtemp2} produces 
\begin{align}
\|\Lambda_{{\theta},h}^{n+1}\|^2_{L^2} 
	&- \|\Lambda_{{\theta},h}^{n}\|^2_{L^2}
	\, + \, \|\Lambda_{{\theta},h}^{n+1} -\Lambda_{{\theta},h}^{n} \|^2_{L^2}
	\, + \, 2\,\kappa \,\Delta t \|\nabla\Lambda_{{\theta},h}^{n+1}\|^2_{L^2}\nonumber\\
	& \, = \,
	2\,\left(\eta_{{\theta}}^{n+1}-\eta_{{\theta}}^{n},\,\, \Lambda_{{\theta},h}^{n+1}\right)
	\, + \, 2\,\kappa\,\Delta t \,\left(\nabla\eta_{{\theta}}^{n+1},\,\, \nabla\Lambda_{{\theta},h}^{n+1}\right)\nonumber\\
	&\, \,\,\,\,\, \,\,\,\,\textcolor{red}{ + \, 2\Delta t\bigg( 
	 b^*(\be_{\bu}^{n}, \theta^{n+1}, \,\Lambda_{\theta,h}^{n+1}) 
	 \, - \, b^*(\be_{\bu}^{n}, \,\, \eta_{\theta}^{n+1}, \,\Lambda_{\theta,h}^{n+1})
	 \, + \, b^*({\bu}^{n}, \eta_{\theta}^{n+1}, \Lambda_{\theta,h}^{n+1})\bigg) }\nonumber\\
	  &\, \,\, \,\,\, \,\,\,\,+ \, 2\,\Delta t E_{2}(\bu, \,\, \theta, \,\, \Lambda_{\theta,h}^{n+1}).\label{errtemp3}
 \end{align}
\, \vspace{2mm}
\textbf{Step 2:} \textbf{[}\textbf{\textit{The estimation of the right hand side terms of error equations.}}\textbf{]}\, \vspace{3mm}\\
\noindent We note that right hand side terms of Equation \eqref{errvel3}-\eqref{errtemp3} are bounded in a similar way. Therefore, we only give estimates of the right hand side terms for Equation \eqref{errvel3}. To bound the first term in \eqref{errvel3}, one can apply the estimate of the dual pairing, and Young's inequality while for the second one the Cauchy-Schwarz and Young's inequalities :
\begin{align*}
2\left(\bfeta_{\bu}^{n+1}-\bfeta_{\bu}^{n},\,\, \bLambda_{\bu,h}^{n+1}\right)
	\, & \leq \,\frac{\Delta t}{\varepsilon_1}\|\bfeta_{\bu,t}\|^2_{L^2(t^n,\, t^{n+1}; H^{-1}(\Omega))}
	 \, + \, {\varepsilon_1}\Delta t\|\nabla\bLambda_{\bu,h}^{n+1}\|^2_{L^2},\\
2\nu\Delta t \left(\nabla\bfeta_{\bu}^{n+1},\,\, \nabla\bLambda_{\bu,h}^{n+1}\right)
	 & \, \leq \, \frac{\, \nu\Delta t}{\varepsilon_2}\| \nabla\bfeta_{\bu}^{n+1}\|^2_{L^2} 
	 + {\varepsilon_2 \nu \, \Delta t}\|\nabla\bLambda_{\bu,h}^{n+1}\|^2_{L^2}.
\end{align*}
For the first non linear term, we first use the error decomposition to get
\begin{align*}
 2\, \Delta t \, 
	 b(\be_{\bu}^{n}, \bu^{n+1}, \,\bLambda_{\bu,h}^{n+1})
	 = 2\, \Delta t \, b(\bfeta_{\bu}^{n}, \bu^{n+1}, \,\bLambda_{\bu,h}^{n+1})
	 -2\, \Delta t \, b(\bphi_{\bu,h}^{n}, \bu^{n+1}, \,\bLambda_{\bu,h}^{n+1}).
\end{align*}
For the first term, apply the second estimate of Lemma~\ref{skewlemma} together with the Young's inequality to obtain
\begin{align*}
 2\, \Delta t \, b(\bfeta_{\bu}^{n}, \bu^{n+1}, \,\bLambda_{\bu,h}^{n+1})
 	& \leq 2\, C\, \Delta t \,\sqrt{\|\bfeta_{\bu}^{n}\|_{L^2}\|\nabla\bfeta_{\bu}^{n}\|_{L^2}}\|\nabla\bu^{n+1}\|_{L^2}\|\nabla\bLambda_{\bu,h}^{n+1}\|_{L^2}\\
 	& \leq \frac{\, C\, \Delta t }{\varepsilon_3}\|\nabla\bu^{n+1}\|^2_{L^2} \|\bfeta_{\bu}^{n}\|_{L^2}\|\nabla\bfeta_{\bu}^{n}\|_{L^2} 
	 + \varepsilon_3 \,\Delta t \|\nabla\bLambda_{\bu,h}^{n+1}\|^2_{L^2}.
	 \end{align*} 
The second term is first expanded by using the definition of $b(\cdot, \cdot, \cdot)$, and then is estimated below by using the H{\"{o}}lder inequality with $L^2-L^{\infty}-L^2$, the Poincar{'{e}}-Friedrich and the Agmon's Inequalities together with Young's Inequality:
\textcolor{red}{
\begin{align*}
 2\, \Delta t 	& \, b(\bphi_{\bu,h}^{n}, \bu^{n+1}, \,\bLambda_{\bu,h}^{n+1}) 
  =  \Delta t \bigg( 
 	 \left(\bphi_{\bu,h}^{n}\cdot\nabla\bu^{n+1}, \,\bLambda_{\bu,h}^{n+1}\right)
 	 - \left(\bphi_{\bu,h}^{n}\cdot\nabla\bLambda_{\bu,h}^{n+1}, \,   \bu^{n+1} \right)
 	 \bigg)\\
 	 \vspace{2cm}
 	 & \leq C \Delta t \left(\,\|\bphi_{\bu,h}^{n}\|_{L^2} \|\nabla\bu^{n+1}\|_{L^{\infty}} \|\bLambda_{\bu,h}^{n+1}\|_{L^2} 
 	 +	\,\|\bphi_{\bu,h}^{n}\|_{L^2} \|\nabla\bLambda_{\bu,h}^{n+1}\|_{L^2}\|\bu^{n+1}\|_{L^{\infty}}\right)
 	 \\
 	  \vspace{2cm}
 	 & \leq C\,\Delta t \,\left(\|\bphi_{\bu,h}^{n}\|_{L^2} \|\nabla\bu^{n+1}\|_{L^{\infty}}C_P \|\nabla\bLambda_{\bu,h}^{n+1}\|_{L^2} 
 	 +	\,\|\bphi_{\bu,h}^{n}\|_{L^2} \|\nabla\bLambda_{\bu,h}^{n+1}\|_{L^2}\|\bu^{n+1}\|_{L^{\infty}} \right)\\
 	 \vspace{1cm}
 	 & \leq C\,\Delta t \,\|\bphi_{\bu,h}^{n}\|_{L^2} \|\bu^{n+1}\|_{\bm H^{3}}\|\nabla\bLambda_{\bu,h}^{n+1}\|_{L^2}\\
 	 \vspace{1cm}
 	 &\leq  \frac{\, C\, \Delta t}{\varepsilon_4} \|\bu^{n+1}\|_{\bm H^3}^2\|\bphi_{\bu,h}^{n}\|_{L^2}^2  
 	 + {\varepsilon _4}\Delta t\,\|\nabla \bLambda_{\bu,h}^{n+1}\|_{L^2}^2.
 \end{align*}
 }
For \textcolor{red}{$b( \bfeta_{\bu}^{n}, \,\, \bfeta_{\bu}^{n+1}, \,\bLambda_{\bu,h}^{n+1})$}, we apply Lemma~\ref{skewlemma} together with Young's inequality, and for \textcolor{red}{$
 b(\bphi_{\bu,h}^{n}, \,\,\bfeta_{\bu}^{n+1}, \bLambda_{\bu,h}^{n+1})$} Lemma~\ref{skewlemma}, the inverse inequality together with Young's Inequality to get
 \textcolor{red}{
\begin{align*}
 2\, \Delta t \, b( \bfeta_{\bu}^{n}, \,\, \bfeta_{\bu}^{n+1}, \,\bLambda_{\bu,h}^{n+1})
	& \leq 2\, C\, \Delta t \sqrt{\|  \bfeta_{\bu}^{n}\|_{L^2}\|\nabla \bfeta_{\bu}^{n+1}\|_{L^2}}\|\nabla \bfeta_{\bu}^{n+1}\|_{L^2} \|\nabla\bLambda_{\bu,h}^{n+1} \|_{L^2} \\
	\vspace{1mm}
	& \leq \frac{\, C \, \Delta t\, }{\varepsilon_5}\| \bfeta_{\bu}^{n}\|_{L^2}\|\nabla \bfeta_{\bu}^{n+1}\|_{L^2}^3
	+ \varepsilon_5\Delta t\|\nabla\bLambda_{\bu,h}^{n+1} \|_{L^2}^2,
	 \\
	 \vspace{1cm}
	 2\, \Delta t \,  b(\bphi_{\bu,h}^{n}, \,\,\bfeta_{\bu}^{n+1}, \bLambda_{\bu,h}^{n+1})\bigg)
	 & \leq 2\, C\, \Delta t \sqrt{\|\bphi_{\bu,h}^{n}\|_{L^2}^{1/2}\|\nabla\bphi_{\bu,h}^{n}\|_{L^2}^{1/2}}\|\nabla\bfeta_{\bu}^{n+1}\|_{L^2}\|\nabla\bLambda_{\bu,h}^{n+1}\|_{L^2}\\	
	 & \leq 2\, C\, \Delta t h^{-1/2}\|\bphi_{\bu,h}^{n}\|_{L^2}\|\nabla\bfeta_{\bu}^{n+1}\|_{L^2}\|\nabla\bLambda_{\bu,h}^{n+1}\|_{L^2}\\ 
	 & \leq \frac{\, C\, \Delta t h^{-1}}{\varepsilon_6}\|\bphi_{\bu,h}^{n}\|_{L^2}^2 \|\nabla\bfeta_{\bu}^{n+1}\|_{L^2}^2 
	 + \varepsilon_6 \Delta t\|\nabla\bLambda_{\bu,h}^{n+1}\|_{L^2}^2 .
\end{align*}
}
\textcolor{red}{In a similar manner, the last nonlinear term is estimated below as follows:
\begin{align*}
2\, \Delta t \, b(\bu^n, \bfeta_{\bu}^{n+1}, \bLambda_{\bu,h}^{n+1})& \leq C \Delta t \|\nabla \bu^n \|_{L^2}\|\nabla\bfeta_{\bu}^{n+1}\|_{L^2}\|\nabla\bLambda_{\bu,h}^{n+1}\|_{L^2}\\
& \leq \frac{C \Delta t}{\varepsilon_7}\|\nabla \bu^n \|_{L^2}^2\|\nabla\bfeta_{\bu}^{n+1}\|_{L^2}^2 \, +  \, \varepsilon_7 \Delta t\|\nabla\bLambda_{\bu}^{n+1}\|_{L^2}^2.
\end{align*}
}
To bound the last two terms, We use error decomposition and apply the Cauchy-Schwarz and Young's inequalities to get
\begin{align*}
2\,\Delta t\, (p^{n+1}- q_h,\,\,\nabla\cdot\bLambda_{\bu,h}^{n+1})
	\, &\leq \, {2\,\Delta t}\|p^{n+1}- q_h\|_{L^2} \|\nabla\bLambda_{\bu,h}^{n+1}\|_{L^2}\\
	\, &\leq \, \frac{\, \Delta t}{\varepsilon_8} \inf\limits_{q_h\in Q_h}\|p^{n+1}- q_h\|_{L^2}^2 + \varepsilon_8 \, \Delta t\|\nabla\bLambda_{\bu,h}^{n+1}\|_{L^2}^2,
\end{align*}
and 
\begin{align*}
2\,\Delta t\, Ri (\langle \bm 0, e_{\widetilde{\theta}}^n \rangle \bLambda_{\bu,h}^{n+1})
	\, &\leq \, {2\,\Delta t}\left(Ri\, \left(\|\eta_{\theta}^{n}\|_{L^2}+\|\Lambda_{\theta,h}^{n}\|_{L^2} \right) \, C_P \|\nabla\bLambda_{\bu,h}^{n+1}\|_{L^2}\right)\\
	\, &\leq \, \frac{ C_P^2 \, Ri^2 \Delta t}{\varepsilon_9} \left(\|\eta_{\theta}^{n}\|_{L^2}^2 +  \|\Lambda_{\theta,h}^{n}\|_{L^2}^2 \right)+ \varepsilon_9 \, \Delta t\|\nabla\bLambda_{\bu,h}^{n+1}\|_{L^2}^2.
\end{align*}
We now bound the terms of $E_{1}(\bu, \theta, \bLambda_{\bu,h}^{n+1})$. First apply Taylor's Theorem with integral remainder term in the first argument of each term of $E_{1}(\bu, \theta, \bv_h)$.  Then use Lemma~\ref{skewlemma}, the Young's inequality for the first and second terms, and the Cauchy-Schwarz, the Poincare and the Young's inequalities for the last term which produce
\begin{align*}
2\,\Delta t\, b(\bu^{n+1}- \bu^{n}, \bu^{n+1}, \bLambda_{\bu,h}^{n+1}) 
   \,& \leq 2\,\Delta t\,\, C \, \Delta t^{1/2} 
   \|\nabla \bu_{t}\|_{L^2(t^n, t^{n+1};L^2(\Omega))}\|\nabla\bu^{n+1}\|_{L^2}\|\nabla\bLambda_{\bu,h}^{n+1}\|_{L^2} \\
   &\leq \,\frac{\,C\, \Delta t^2}{\varepsilon_{10} }\|\nabla \bu_{t}\|_{L^2(t^n, t^{n+1};L^2(\Omega))}^2 \|\nabla\bu^{n+1}\|_{L^2}^2 
   \, + \,\varepsilon_{10} \Delta t \|\nabla\bLambda_{\bu,h}^{n+1}\|_{L^2}^2,
\end{align*}
and
\begin{align*}
2\,\Delta t\, Ri \left(\langle 0, \,\theta^{n+1} - \theta^{n} \rangle, \,\bLambda_{\bu,h}^{n+1}\right)
& 
\leq 2\,\Delta t\, Ri \|\theta^{n+1} - \theta^{n}\|_{L^2}C_P \|\nabla\bLambda_{\bu,h}^{n+1}\|_{L^2}\\
& \leq \frac{C_P^2\, Ri^2 \Delta t^2}{\varepsilon_{11}} \|\theta_t\|^2_{L^2(t^n, t^{n+1};L^2(\Omega))}
	\, + \,\varepsilon_{11} \Delta t\|\nabla\bLambda_{\bu,h}^{n+1}\|_{L^2}^2,
	\\
2\,\Delta t\,\left(\bu_{t}^{n+1}- \frac{{\bu}^{n+1} -\bu^n}{\Delta t} , \,\,\bLambda_{\bu,h}^{n+1}\right)
   \, & \leq 2\,\Delta t \, \Delta t^{1/2} \|\bu_{tt}\|_{L^2(t^n, t^{n+1};L^2(\Omega))}\, C_P \, \|\nabla\bLambda_{\bu,h}^{n+1}\|_{L^2}\\
   \,&\leq \, \,\frac{C_P^2 \, \Delta t^2}{ \varepsilon_{12}}\,\|\bu_{tt}\|_{L^2(t^n, t^{n+1};L^2(\Omega))}^2
    \, + \, \varepsilon_{12} \Delta t\|\nabla\bLambda_{\bu,h}^{n+1}\|_{L^2}^2 .
\end{align*}
Plugging these estimates into $E_{1}(\bu, \theta, \bLambda_{\bu,h}^{n+1})$ yields 
\begin{gather*}
2\,\Delta t\, E_{1}(\bu, \theta, \bLambda_{\bu,h}^{n+1}) \\
\, \leq \, { \Delta t^2}\left(
\,  \frac{C}{\varepsilon_{10}}\|\nabla\bu^{n+1}\|_{L^2}^2 \|\nabla \bu_{t}\|_{L^2(t^n, t^{n+1};L^2(\Omega))}^2 
 + \frac{C_P^2 \,Ri^2}{\varepsilon_{11}}\|\theta_t\|_{L^2(t^n, t^{n+1};L^2(\Omega))}^2
 \, + \,\frac{C_P^2 \,}{\varepsilon_{12}} \|\bu_{tt}\|_{L^2(t^n, t^{n+1};L^2(\Omega))}^2
	\right)\\
   \, + \, \left( \varepsilon_{10} + \varepsilon_{11} + \varepsilon_{12}\right)\,  \Delta t \, \|\nabla\bLambda_{\bu,h}^{n+1}\|_{L^2}^2.
\end{gather*}
\textbf{Step 3:} \textbf{[}\textbf{\textit{The application of the Gronwall Lemma.}} \textbf{]}\, \vspace{2mm}\\
Insert these bounds on the right hand side of \eqref{errvel3} along with the appropriate choice of $\varepsilon_i, i=1,...,12$. Then using Lemma~\ref{yarlemmaBouss} produces \\
\begin{align}
	& \|\bphi_{\bu,h}^{n+1}\|^2_{L^2}  - \|\bphi_{\bu,h}^{n}\|^2_{L^2}
	\, + \, \beta\left( \|\nabla\cdot\bphi_{\bu,h}^{n+1}\|^2_{L^2} - \|\nabla\cdot\bphi_{\bu,h}^{n}\|^2_{L^2}\right)
	\, + \, \|\bLambda_{\bu, h}^{n+1} - \bphi_{\bu,h}^{n+1}\|^2_{L^2}
	\, + \, \|\bLambda_{\bu,h}^{n+1} - \bphi_{\bu,h}^{n} \|^2_{L^2}
	\nonumber\\
	\vspace{10mm}
	 & \,  \, \, \,\,\,\,\,\, \,\hspace{14mm} + \, \frac{\beta}{2}\|\nabla\cdot(\bphi_{\bu,h}^{n+1}-\bphi_{\bu,h}^{n})\|^2_{L^2}
	  \, + \, \gamma\, \Delta t \, \|\nabla\cdot \bphi_{\bu,h}^{n+1}\|^2_{L^2} 
	\, + \,\nu\, \Delta t\, \|\nabla\bLambda_{\bu,h}^{n+1}\|^2_{L^2}
	\nonumber\\
	\vspace{22mm}
	&   \,    \leq  
	{C\, \nu^{-1}}\, \Delta t \, \|\bfeta_{\bu,t}\|^2_{L^2(t^n,\, t^{n+1}; H^{-1}(\Omega))}
	\, + C\,\Delta t \left(\, \nu + \nu^{-1} \|\nabla\bu^n\|^2\,  + \, \gamma \,\right)\,\| \nabla\bfeta_{\bu}^{n+1}\|^2_{L^2}\, \,\,\, \vspace{2cm}
	\nonumber\\
	\vspace{22mm}
	 & \, \,\,\, + \, {C\, \nu^{-1}\, \Delta t\, }\|\nabla\bu^{n+1}\|^2_{L^2} \|\bfeta_{\bu}^{n}\|_{L^2}\|\nabla\bfeta_{\bu}^{n}\|_{L^2}   
	\, + \,  C\,\nu^{-1}\, \Delta t \,\bigg(\|\bu^{n+1}\|^2_{\bm H^3(\Omega)} + h^{-1}\|\nabla\bfeta_{\bu}^{n+1}\|^2\bigg)\|\bphi_{\bu,h}^{n}\|_{L^2}^2
	\nonumber\\
	\vspace{22mm}
	 & \,\,\,\,  + \, { C\, \nu^{-1} \, \Delta t}\| \bfeta_{\bu}^{n}\|_{L^2}\|\nabla \bfeta_{\bu}^{n+1}\|_{L^2}^3
	 \, + \, C\,\nu^{-1}\, \Delta t \, \inf\limits_{q_h\in Q_h}\|p^{n+1}- q_h\|_{L^2}^2 			\nonumber\\
	 \vspace{15mm}
	 & \, \,\,\,  +\,C\,Ri^2 \, C_P^2 \,\nu^{-1}\,\Delta t\,   \bigg( \, \|\eta_{\theta}^{n}\|_{L^2}^2 + \|\Lambda_{\theta,h}^{n}\|_{L^2}^2 \, \bigg)	 
	  \, + \,  \beta\,(1 + 2\,\Delta t\,)\|\nabla\bfeta_{\bu, t}\|^2_{L^{2}(t^n, t^{n+1}; L^2(\Omega))}
	  \nonumber\\
	  \vspace{4mm}
	  & \, \,\,\, + \, {C \, \nu^{-1} \, \Delta t^2}\left(
	   \|\nabla\bu^{n+1}\|_{L^2}^2 \, \|\nabla\bu_t\|^2_{L^2(t^n, t^{n+1};L^2(\Omega))}
	   \, + \, C_P^2\, \|\bu_{tt}\|^2_{L^2(t^n, t^{n+1};L^2(\Omega))}
	  \, + \, C_P^2 \, Ri^2\,\|\theta_{t}\|^2_{L^2(t^n, t^{n+1};L^2(\Omega))}
	  \right)
	  \nonumber\\
	 \vspace{10mm}
	 & \,\, \, \,  + \, d\, \beta \, (\, 1 + \, 2 \, \Delta t \, )\|\nabla\bfeta_{\bu, t}\|^2 \, + \,\beta\, \Delta t\|\nabla\cdot\phi_{\bu,h}^n\|^2_{L^2}
	 \label{errvel4} .
	\end{align}
Similarly, we bound the terms on the right hand side of \ref{errtemp3} as follows
\begin{multline}
\|\Lambda_{{\theta},h}^{n+1}\|^2_{L^2} - \|\Lambda_{{\theta},h}^{n}\|^2_{L^2}
	\, + \, \|\Lambda_{{\theta},h}^{n+1} - \Lambda_{{\theta},h}^{n} \|^2_{L^2}
	\, + \, \kappa \,\Delta t \|\nabla\Lambda_{{\theta},h}^{n+1}\|^2_{L^2}\\
	\leq  
	 C\, \kappa^{-1}\Delta t\, \|\eta_{\theta,t}\|^2_{L^2(t^n,\, t^{n+1}; H^{-1}(\Omega))}  
	 \, + C \, \kappa \,\Delta t \,\| \nabla\bfeta_{\theta}^{n+1}\|^2_{L^2}
	\, + \, C\, \kappa^{-1}\, \Delta t\,\|\nabla\theta^{n+1}\|^2_{L^2} \|\bfeta_{\bu}^{n}\|_{L^2}\|\nabla\bfeta_{\bu}^{n}\|_{L^2} \\
	\, + \, C \, \kappa^{-1}\, \Delta t\, \bigg(\|\theta^{n+1}\|_{H^3}^2 + h^{-1}\|\nabla\bfeta_{\theta}^{n+1}\|_{L^2}^2 \bigg)\|\bphi_{\bu,h}^{n}\|_{L^2}^2 
 	\,+ \,C \kappa^{-1} \Delta t\|\bfeta_{\bu}^{n}\|_{L^2}\|\nabla\bfeta_{\bu}^{n}\|_{L^2}\|\nabla \eta_{\theta}^{n+1}\|_{L^2}^2
 	 \\
 	\, + \, C \kappa^{-1} \Delta t\|\nabla\bu^{n}\|_{L^2}^2 \|\nabla\eta_{\theta}^{n+1}\|_{L^2}^2 
 	\, + \,{C\, \kappa^{-1}}\Delta t^2\bigg(\|\theta_{tt}\|^2_{L^2(t^n, t^{n+1};L^2(\Omega))}
	  \, + \, \|\nabla\theta^{n+1}\|_{L^2}^2 \, \|\nabla\bu_t\|^2_{L^2(t^n, t^{n+1};L^2(\Omega))}\bigg).
	  \label{errtemp4}
 \end{multline}
Drop the non-negative the fourth, fifth and sixth left hand side terms of \eqref{errvel4} and the third left hand side term on \eqref{errtemp4}. Next use the regularity assumptions on Boussinesq solution, and sum over time steps. This produces
\begin{align}
	&\|\bphi_{\bu,h}^{N}\|^2_{L^2} 
	\,  + \, \beta \|\nabla\cdot\bphi_{\bu,h}^{N}\|^2_{L^2} 
	 \, +  \, \Delta t\sum_{n=0}^{N-1}\left(\gamma\|\nabla\cdot \bphi_{\bu,h}^{n+1}\|^2 _{{L^2}}
	\, + \,\nu \|\nabla\bLambda_{\bu,h}^{n+1}\|^2_{{L^2}} \right)\nonumber\\
	& \leq  
	 C \,\Delta t \sum_{n=0}^{N-1} \bigg[\nu^{-1}
	\, \left( \, \|\bu^{n+1}\|^2_{H^3} + h^{-1}\|\nabla \bfeta_{\bu}^{n+1}\|_{L^2}^2\, \right) \|\bphi_{\bu,h}^{n}\|_{L^2}^2
	 \, + \,  \beta   \|\nabla\cdot\bphi_{\bu,h}^{n}\|_{L^2}^2
	 + Ri^2 \, C_P^2 \, \|\Lambda_{\theta,h}^{n} \|^2_{L^2}\bigg] \nonumber\\
	 &  + C \, \Delta t \sum_{n=0}^{N-1}\bigg[\nu^{-1} \|\eta_{\bu, t}\|^2_{L^2(t^n, t^{n+1}; H^{-1}(\Omega))}
	\, + \,\big(\,\nu \, + \nu^{-1} \, \|\nabla\bu\|_{L^{\infty}(0, T; L^2(\Omega))}^2\, + \gamma \, \big)\|\nabla\eta_{\bu}^{n+1}\|_{L^2}^2 \nonumber\\
	& \,  + \, \, \nu^{-1}\|\nabla\bu\|_{L^{\infty}(0, T; L^2(\Omega))}^2 \|\nabla\eta_{\bu}^n\|^2 
	+ \nu^{-1}\, \| \bfeta_{\bu}^{n}\|_{L^2}\|\nabla \bfeta_{\bu}^{n+1}\|_{L^2}^3
	\, + \,\nu^{-1}\, Ri^2 \, C_P^2 \, \|\eta_{\theta}^n\|^2 + \,\, \nu^{-1} \inf\limits_{q_h\in Q_h}\|p^{n+1}- q_h\|_{L^2}^2 \bigg]\nonumber\\
	  & \,  + \, {C \, \nu^{-1} \, \Delta t^2}\left(
	   \|\nabla\bu\|_{L^{\infty}(0, T; L^2(\Omega))}^2 \, \|\nabla\bu_t\|^2_{L^2(0, T; L^2(\Omega))}
	   \, + \, C_P^2\, \|\bu_{tt}\|^2_{L^2(0, T;L^2(\Omega))}
	  \, + \, C_P^2 \, Ri^2\,\|\theta_{t}\|^2_{L^2(0, T; L^2(\Omega))}
	  \right)
	  \nonumber\\
	  \, &  + C\sum_{n=0}^{N-1} \, d \, \beta\, (\, 1 \, + 2 \, \Delta t\, )\|\nabla\eta_{\bu, t}\|^2_{L^2} \, 
	 \, +  \, \|\bphi_{\bu,h}^{0}\|^2_{L^2} 
	\, + \, \beta \|\nabla\cdot\bphi_{\bu,h}^{0}\|^2_{L^2}
	  \label{errvel5},
\end{align}
and
\begin{align}
\|\Lambda_{{\theta},h}^{N}\|^2_{L^2} & + \, \kappa \,\Delta t \sum_{n=0}^{N-1}\|\nabla\Lambda_{{\theta},h}^{n+1}\|^2_{L^2}\nonumber\\
	\leq 
	& \, C \Delta t\sum_{n=0}^{N-1}\kappa^{-1} \bigg[\|\theta^{n+1}\|_{\bm H^3}^2 
 	\, + \, h^{-1} \, \|\nabla\eta_{\theta}^{n+1}\|_{L^2}^2\bigg]\|\bphi_{\bu,h}^{n}\|_{L^2}^2 \nonumber\\
	& \, + \,C \Delta t\sum_{n=0}^{N-1}\bigg(\kappa^{-1}\|\eta_{t,\theta} \|^2_{L^2(t^n, t^{n+1}, H^{-1}(\Omega))} + \kappa\|\nabla\eta_{\theta}^{n+1}\|^2_{L^2} 
		+ \kappa^{-1}\|\nabla\theta\|_{L^{\infty}(0, T; L^2(\Omega))}^2 \|\nabla\bfeta_{\bu}^{n}\|_{L^2}^2 
		\nonumber\\
		&\, \, \, \,\,\,\, \hspace{20mm} + \kappa^{-1}\|\bfeta_{\bu}^{n}\|_{L^2}\|\nabla\bfeta_{\bu}^{n}\|_{L^2}\|\nabla\eta_{\theta}^{n+1}\|_{L^2}^2
 	\, + \,  \kappa^{-1}\, \|\nabla{\bu}^{n}\|_{L^2}^2 \|\nabla\bfeta_{\theta}^{n+1}\|_{L^2}^2 \bigg)
 	\nonumber\\
 	& \, + \,{C\, \kappa^{-1}}\Delta t^2 
 	\left( 
 		\|\nabla\theta\|_{L^{\infty}(0, T; L^2(\Omega))}^2 \, \|\nabla\bu_t\|^2_{L^2(0, T; L^2(\Omega))}
	   \, + \,  \|\theta_{tt}\|^2_{L^2(0, T;L^2(\Omega))}
 	\right) 
 	\, + \, \|\Lambda_{{\theta},h}^{0}\|^2_{L^2}.\label{errtemp5}
 \end{align}
Add \eqref{errtemp5} to \eqref{errvel5}, \textcolor{red}{assume that $h\leq 1$} and notice that $\bphi_{\bu,h}^{0}= 0,\, \,\Lambda_{\theta,h}^{0}= 0.$ Then apply Gronwall Lemma which yields:
\vspace{1mm}
\begin{align}
	\|\bphi_{\bu,h}^{N}\|^2_{L^2}  + \, \beta \|\nabla\cdot\bphi_{\bu,h}^{N}{\|}^2_{L^2} \, + \, \|\Lambda_{{\theta},h}^{N}\|^2_{L^2}
	\, + \, \gamma\||\nabla\cdot \bphi_{\bu,h}|\|^2_{2,0} 
	\, + \,\nu \||\nabla\bLambda_{\bu,h}|\|^2_{2,0} 
	+ \, \kappa  \||\nabla\Lambda_{{\theta},h}|\|^2_{2,0}
	\nonumber\\
	\leq
	C \, \bigg[ (\, \nu^{-1}+ \kappa^{-1} \, ) h^{2k+2} + (\, \nu  + \nu^{-1} +\kappa + \kappa^{-1}\,)\, h^{2k} + (\nu^{-1}+ \kappa^{-1}) h^{4k+1} +	 \beta\,(1 + 2\,\Delta t\,)h^{2k} +(\nu^{-1}+ \kappa^{-1})\Delta t^2\bigg]
\label{gronwallapplied}.
\end{align}
\textbf{Step 4:}  \textbf{[}\textbf{\textit{The completion of proof.}} \textbf{]} \, \vspace{1mm}\\
The application of the triangle inequality to all error terms gives
\begin{gather}
\|e_{\bu}^{N}\|^2_{L^2} 
	\, + \, \|e_{\widetilde{\theta}}^{N}\|^2_{L^2}
	\, + \, \beta \|\nabla\cdot \be_{\bu}^{N}\|^2_{L^2} 
	\, + \, \gamma\||\nabla\cdot \be_{\bu}|\|^2_{2,0}
	\, + \,\nu \||\nabla \be_{\widetilde{\bu}}|\|^2_{2,0}
	\, + \,	\kappa\||\nabla e_{\widetilde{\theta}}|\|^2_{2,0}\nonumber\\
	\, \leq \,
	2\bigg(\|\bfeta_{\bu}^{N}\|^2_{L^2}  
	\, + \, \|\eta_{\theta}^{N}\|^2_{L^2}
	\, + \, \beta \|\nabla\cdot\bfeta_{\bu}^{N}\|^2_{L^2} 
	\, + \, \gamma\||\nabla\cdot \bfeta_{\bu}|\|^2_{2,0}
	\, + \,\nu\||\nabla\bfeta_{\bu}|\|^2_{2,0}
	\, + \,	\kappa\,\||\nabla\eta_{\theta}|\|^2_{2,0}
	 \bigg) \nonumber\\
	\, + \, 2\, \bigg( \|\bphi_{\bu,h}^{N}\|^2_{L^2}
	\, + \,\|\Lambda_{\theta,h}^{N}\|^2_{L^2}
	\, + \, \beta \|\nabla\cdot\bphi_{\bu,h}^{N}\|^2_{L^2}
	\, + \, \gamma\||\nabla\cdot \bphi_{\bu,h}|\|^2_{2,0}
	\, + \,\nu\||\nabla\bLambda_{\bu,h}|\|^2_{2,0}
	\, + \,	\kappa\||\nabla\Lambda_{\theta,h}|\|^2_{2,0}\nonumber
	 \bigg). 
\end{gather}
Finally, using the regularity assumptions on the Boussinesq solutions, approximation properties, and estimate \eqref{gronwallapplied} finishes the proof.
\end{proof}
\section*{Numerical Experiments}
This section presents two numerical experiments to test the predicted convergence rates of the previous section, and illustrate the reliability of Algorithm~\ref{algBouss}. All tests are implemented using FreeFem++ \cite{Freefem}. 
\subsection{Convergence Rate Verification} 
With the use of the finite element spaces $(\bm P_2, P_1, P_2)$ for the velocity/pressure/temperature, respectively, Theorem~\ref{convBouss} predicts second order convergence in space. To illustrate this, we choose a test problem with analytical solutions
\begin{align*}
u(x, t)&=\left[%
\begin{array}{c}
\cos(\pi (y-t)) \nonumber\\
\sin(\pi (x+t))%
\end{array}%
\right]\exp(t), \hspace{2mm}
p(x, t)=\sin(x+y)(1+t^2), \hspace{2mm}T(x, t)=\sin(\pi x)+y\exp(t),
\end{align*}
on the unit square $(0,1)^2$ with $\nu = 1,\,\, Ri =1,\,\, \kappa=1,$ and stabilization parameters $\gamma=1.0$ and $\beta=1.0$. The forcing terms $\bf, \Psi$ are calculated from $\bu, p, \theta$ and the Boussinesq equations. Then, we compute solutions to Algorithm~\ref{algBouss} on a series of refined mesh by choosing small end time $T = 0.001$ and time step $\Delta t=0.0001$ in order to isolate the spatial errors.\\
The results are presented in Table~\ref{BousSpatialVelErr}, and are consistent with our theoretical findings for convergence rates. \\
We next test the predicted temporal rates using the same test problem. We run our method on a fixed mesh with mesh size $h=1/64$ with a series of timestep sizes. The computed errors and rates are presented in Table~\ref{BousTemporalVelErr}, and we observe the predicted optimal rates. Note that the divergence error is always small here, due to the finer mesh and the stabilization. Since this error is already as small as linear solver error, we do not expect convergence rates (as it has already converged).
\begin{table}[h!]
\centering
\captionsetup{width=.9\textwidth, font=scriptsize}
\caption{\label{BousSpatialVelErr} \small Spatial velocity errors and rates for a fixed end time ${T=0.001}$, a time step ${\Delta t=0.0001}$.}
 \begin{adjustbox}{width=0.9\textwidth}
  \begin{tabular}{ c  cl  ll  ll  ll }
  \hline
    \hline
    $h$ & $\||\bu -\bu_h\||_{\infty, 0}$ & Rate & $\||\nabla\cdot(\bu -\bu_h)\||_{\infty, 0}$ & Rate & $\||\nabla\cdot(\bu -\bu_h)\||_{2, 0}$ & Rate & $\||\nabla({\bu} -\widetilde{\bu}_h)\||_{2, 0}$ & Rate \\ 
    \hline 
    $ 1/4 $ & $ 1.9978e-2 $ & -- & $ 1.9001e-5 $ & -- & $ 1.1808e-7 $ & -- & $ 2.7976e-3 $ & -- \\
     
    $ 1/8 $ & $ 2.5644e-3 $ & 2.9618 & $ 5.5915e-6  $ & 1.7648 & $ 3.4854e-8 $ & 1.7604 & $ 7.1600e-4 $ & 1.9662 \\
   
    $ 1/16 $ & $ 3.2267e-4 $ & 2.9904 & $ 1.2480e-6 $ & 2.1636 & $ 7.7778e-9  $ & 2.1639 & $ 1.7974e-4 $ & 1.9940 \\
    
    $ 1/32 $ & $ 4.0400e-5 $ & 2.9976 & $ 2.8658e-7 $ & 2.1226 & $ 1.7864e-9 $ & 2.1222 & $ 4.4702e-5 $ & 2.0075  \\
     
    $ 1/64 $ & $ 5.0521e-6 $ & 2.9994 & $ 6.0716e-8 $ & 2.2388 & $ 3.8071e-10 $ & 2.2388 & $ 1.0964e-5 $ & 2.0274 \\
    \hline
    \hline 
  \end{tabular}
 \end{adjustbox}
\end{table} 

\begin{table}[h!]
\centering
\captionsetup{width=.9\textwidth, font=scriptsize}
\caption{\label{BousTemporalVelErr} \small  Temporal velocity errors and rates on a fixed mesh size $h=1/64$.}
 \begin{adjustbox}{width=.9\linewidth}
  \begin{tabular}{ c  cl  ll  ll  ll }
    \hline
    \hline
    $\Delta t$ & $\||\bu -\bu_h\||_{\infty, 0}$ & Rate & $\||\nabla\cdot(\bu -\bu_h)\||_{\infty, 0}$ & Rate & $\||\nabla\cdot(\bu -\bu_h)\||_{2, 0}$ & Rate & $\||\nabla({\bu} - \widetilde{\bu}_h)\||_{2, 0}$ & Rate \\ 
    \hline 
    $ {1}/4 $ & $ 4.1914e-2 $ & --- & $ 3.3464e-8 $ & --- & $ 2.1720e-8 $ & --- & $ 2.4817e-1 $ & ---  \\ 
    $ {1}/8 $ & $ 2.7726e-2 $ & 0.5962 & $ 3.6204e-8 $ &--- & $ 2.1113e-8 $ &---  & $ 1.4157e-1 $ &0.8098 \\ 
    $ {1}/16 $ & $ 1.5418e-2 $ & 0.8467 & $ 4.3557e-8 $ & --- & $ 2.5392e-8 $ & --- & $ 7.3314e-2 $ & 0.9494 \\ 
    $ {1}/32 $ & $ 8.0705e-3 $ & 0.9339 & $ 6.3380e-8 $ & --- & $ 3.8401e-8 $ & --- & $ 3.7046e-2 $ & 0.9848 \\
    $ {1}/64$ & $ 4.1217e-3 $ & 0.9694 & $ 1.0736e-8 $ & --- & $ 6.4309e-8 $ & --- & $ 1.8595e-2 $ & 0.9944 \\
    \hline
    \hline
  \end{tabular}
  \end{adjustbox}
  \label{tablevel1}
\end{table}
\subsection{Error comparison for a test problem with larger pressure.}
Our second numerical experiment focuses on the pressure robustness of the proposed algorithm. One important advantage of grad-div stabilization is that the stabilization parameter $\gamma$ with the appropriate selection reduces the negative impact of the continuous pressure on the velocity error. To test this for the proposed method, a similar test problem and set up for the $2d$-convergence rate test are used, but fixing end time and time step to $T=0.01$, $\Delta t=T/8$ We take the dimensionless kinematic viscosity $Pr = 1.0, Ra=100$ and true pressure solution as 
$$\,p(x,y)= 1000 \sin ({x+2y}),$$
\textcolor{red}{and varying the stabilization parameters $\beta=0, 0.2, 1.0$. Then, we run Algorithm~\ref{algBouss} and the standard grad-div stabilization method for varying $\gamma=10^{k}, k=-1,\,  0,\, 1,\, ...,\, 5$ on successively refined meshes. Our calculations reveal that both methods give (quasi-) optimal errors when $\gamma=10^5$, $\, \beta=0$. The results are presented in Table~\ref{IncPres1}, and we observe that with quasi-optimal parameter choices, the proposed method performs just as well as grad-div stabilization with $L^2(H^1)$-errors and provides better divergence errors.}
\begin{table}[h!]
\centering
\captionsetup{width=\textwidth, font=scriptsize}
\caption{ Velocity errors and divergence of the non-stabilized, the standard grad-div and modular grad-div methods with $\gamma=10^5$, $\beta = 0.0$ for large pressure.}
\label{IncPres1}
\begin{adjustbox}{width=0.9\textwidth}
\begin{tabular}{l l l l l l l l l l}
\hline
\hline
\multirow{2}{1.5cm}{}&\multicolumn{3}{p{5cm}}{ $\||\nabla(\bu -\bu_h)\||_{2, 0}$ }&\multicolumn{3}{p{5cm}}{ $\||\nabla\cdot(\bu -\bu_h)\||_{2, 0}$ }&\multicolumn{3}{p{5cm}}{ $\|\nabla\cdot\bu_h^N\|_{L^2}$ }\bigstrut 
\\
$h$ & No-stab.& Standard & Modular & No-stab.& Standard & Modular & No-stab.& Standard & Modular
\\
\hline
1/2 & 0.1827 & 2.8140e-2  &  2.8138e-2  & 0.1170 & 5.5878e-6  & 4.3595e-6  & 1.5534 & 5.5895e-5  & 4.3610e-5
\\
1/4  & 0.1734  & 7.2051e-3 & 7.3944e-3  & 0.1676  & 3.2778e-6 & 2.0013e-6 &  1.97535 & 3.2777e-5 & 2.0040e-5
\\
1/8  & 3.1000e-2 & 1.8125e-3  & 1.8811e-3 & 0.0303 & 2.1002e-6 & 1.2088e-7 & 0.3136 & 2.1002e-5  & 1.2139e-6
\\
1/16 & 4.3272e-3 & 4.5537e-4 & 4.5952e-4 & 4.2361e-3 &   2.0248e-6 & 6.2031e-9 & 4.2675e-2 & 2.0248e-5  & 6.2120e-8
\\
1/32 & 5.7400e-4 & 1.1995e-4 & 1.1987e-4 & 5.5360e-4 &   2.0205e-6 & 8.0417e-10 & 5.5457e-3 & 2.0205e-5  & 8.0234e-9
\\
\hline
\hline
\end{tabular}
\end{adjustbox}
\end{table}

\subsection{Error comparison on a fixed mesh with varying Rayleigh numbers}
\textcolor{red}{
We next compare the velocity errors of non-stabilized, standard grad-div and modular grad-div methods by fixing the stabilization parameters $\beta=0$ with varying Rayleigh numbers, $Ra=10^k, k=-1,\, 0, \, 1, \, ..., \, 5$. We use the same velocity, pressure and temperature field solutions and set-up as for $2d$ convergence
rate verification. We fix end time and mesh size to $h=1/32, \,\,\Delta t = 0.1/32.$ The computed errors from these methods reveals that both two methods again give optimal errors and mass conservations for $\gamma=10^5, \beta =0$. From Table~\ref{IncRa1}, we observe that modular grad-div performs better than the standard grad-div with quasi-optimal parameters.}

\begin{table}[h!]
\centering
\captionsetup{width=\textwidth, font=scriptsize}
\caption{ Velocity errors and divergences of the non-stabilized, the standard grad-div and modular grad-div methods for $\gamma=10^5, \beta=0$ with varying $Ra$.}
\label{IncRa1}
\begin{adjustbox}{width=0.9\textwidth}
\begin{tabular}{l l l l l l l l l l}
\hline
\hline
\multirow{2}{1.0cm}{}&\multicolumn{3}{p{5cm}}{ $\||\nabla(\bu -\bu_h)\||_{2, 0}$ }&\multicolumn{3}{p{5cm}}{ $\||\nabla\cdot(\bu -\bu_h)\||_{2, 0}$ } &\multicolumn{3}{p{5cm}}{$\|\nabla\cdot\bu_h\|_{L^2}$}\bigstrut 
\\
$Ra$ & No-stab.& Standard & Modular & No-stab.& Standard & Modular & No-stab.& Standard & Modular
\\
\hline
1  & 8.2819e-4  & 8.2655e-4 & 9.7395e-5 & 3.5958e-6 & 2.2774e-8 & 1.2521e-10 & 1.2409e-5  & 8.238e-8  & 4.5791e-10 
\\
$10$  & 8.8942e-4  & 8.8772e-4 & 1.0701e-4 & 3.6836e-6 & 2.4118e-8  & 1.3331e-10  & 1.2867e-5 & 8.8923e-8 &  4.9666e-10
\\
$10^2$ & 1.5607e-3  & 1.5582e-3  & 2.0482e-4 & 4.8806e-6  & 4.1537e-8 & 2.4328e-10 & 1.8964e-5 & 1.7033e-7 &  1.0027e-9
\\
$10^3$  & 8.2600e-3  & 8.2498e-3  &  1.1175e-3 &  2.1148e-5 &  2.3428e-7 & 1.4569e-9 & 9.0468e-5  & 9.8050e-7  & 6.0459e-9 
\\
$10^4$  & 4.3291e-2  & 4.3255e-2 & 5.4819e-3  & 1.4816e-4 & 1.2905e-6 & 8.0101e-9 & 5.3082e-4 & 4.3849e-6 & 2.7010e-8
\\
$10^5$ & 0.2714 & 2.6895e-1 &  2.2405e-2 & 2.3552e-3  & 1.1264e-5  &  7.5111e-8  &  7.8249e-3 &  3.0637e-5  & 1.9085e-7 
\\
\hline
\hline
\end{tabular}
\end{adjustbox}
\end{table}
%
%

\subsection{Mass conservation and error comparison for different finite element choice with varying $\gamma$}
\textcolor{red}{In this section, we compare errors and the mass conservation of the non-stabilized, the usual grad-div and modular grad-div stabilization by using different finite element spaces for fixed Rayleigh number, $Ra=10^6$ and varying $\gamma=10^k, \,\,k=0,1,2,...,5$. First, we compute solutions for $(\bP_2, P_1, P_2), $ and $(\bP_2, P_0, P_2), $ on the barycenter (bc) refinement of $[0,1]\times [0,1]$, which is created by a $8\time8$ uniform mesh. Next we repeat the calculations for $(\bP_2, P_1, P_2), $ and $(\bP_2, P_0, P_2), $ and $(\bP_{1b}, P_1, P_{1b}).$ on s non barycentred refined mesh. The results are presented in Table~\ref{IncRa2}, and reveals that the proposed method gives much more accurate solutions, especially on non barycentered meshes.}

\begin{table}[h!]
\centering
\captionsetup{width=.9\textwidth, font=scriptsize}
\caption{ Velocity errors and divergences of the non-stabilized, the standard grad-div and modular grad-div methods with varying $\gamma$.}
\label{IncRa2}
\begin{adjustbox}{width=0.9\textwidth}
\begin{tabular}{l l l l l l l l l l }
\hline
\hline
\multirow{2}{0.0cm}{}&\multicolumn{2}{p{3cm}}{  }&\multicolumn{2}{p{4cm}}{ $\||\nabla(\bu-\bu_h)\||_{2,0}$}&\multicolumn{2}{p{4cm}}{ $\||\nabla\cdot(\bu-\bu_h)\||_{2,0}$} &\multicolumn{2}{p{5cm}}{$\|\nabla\cdot\bu_h\|_{L^2}$}\bigstrut 
\\
Element & Mesh & $\gamma$ & Standard & Modular & Standard & Modular & Standard & Modular
\\
\hline
$(\bP_2, P_1, P_2)$ & Bc  & $0$ & 3.1421e-2 &  - - - -& 1.7557e-2 &  - - - -& 0.6894  & - - - -  
\\
$(\bP_2, P_1, P_2)$ & Bc  & $10^0$ & 2.5559e-2 & 3.1477e-4 & 1.2739e-2 &   1.0255e-2 & 0.4931  & 0.3962  
\\
$(\bP_2, P_1, P_2)$ & Bc  & $10$  & 1.3431e-2 & 1.9548e-4 & 4.5858e-3 &   3.0573e-3 & 0.1692  & 0.1117 
\\
$(\bP_2, P_1, P_2)$ & Bc  & $100$ & 8.7456e-3 & 1.6550e-4 & 6.9048e-4  & 4.2361e-4  & 2.4573e-2 & 1.4975e-2
\\
$(\bP_2, P_1, P_2)$ & Bc  & $1000$ & 8.5687e-3 & 1.6381e-4 & 7.3200e-5 & 4.8875e-5  & 2.5907e-3 & 1.7169e-3 
\\
$(\bP_2, P_1, P_2)$ & Bc  & $10000$ & 8.5651e-3 & 1.6360e-4 & 7.3768e-6  & 5.7394e-6  & 2.6052e-4  & 2.0159e-4
\\
$(\bP_2, P_1, P_2)$ & Bc  & $100000$ & 8.5650e-3 & 1.6358e-4 & 7.3996e-7 & 5.9433e-7  & 2.6118e-5 & 2.0877e-5 
\\
\\
$(\bP_2, P_0, P_2)$ & Bc  & $0$ & 3.6389e-2 &   - - - -& 2.5639e-2  &    - - - - & 1.0576 & - - - -  
\\
$(\bP_2, P_0, P_2)$ & Bc  & $1$ & 2.8158e-2  & 4.6986e-4 & 1.8064e-2 &    1.4223e-2 & 0.7092 & 0.5667 
\\
$(\bP_2, P_0, P_2)$ & Bc  & $10$ & 1.3443e-2 & 3.0108e-4 & 5.5204e-3  &   3.8988e-4& 0.1986 & 0.1466
\\
$(\bP_2, P_0, P_2)$ & Bc  & $100$  & 8.7021e-3 & 2.3387e-4 & 7.6563e-4 &   7.8538e-4& 2.6879e-2 & 2.8369e-2
\\
$(\bP_2, P_0, P_2)$ & Bc  & $1000$ & 8.5647e-3 & 1.7950e-4  & 8.0258e-5  & 2.0423e-4 &  2.8069e-3 & 6.9592e-3 
\\
$(\bP_2, P_0, P_2)$ & Bc  & $10000$ & 8.5648e-3 & 1.6786e-4 & 8.0659e-6 & 2.8254e-5 &  2.8197e-4 & 9.4209e-4 
\\
$(\bP_2, P_0, P_2)$ & Bc  & $100000$ & 8.5649e-3 & 1.6745e-4 & 8.0865e-7 & 2.9466e-6 & 2.8257e-5 &  9.7958e-5
\\
\\
$(\bP_2, P_1, P_2)$ & Nbc  & $0$ & 1.7543e-2 &- - - -   & 0.1116 & - - - -  & 2.8909e-3 &  - - - - 
\\
$(\bP_2, P_1, P_2)$ & Nbc & $1$ & 1.7315e-2 &  1.0195e-5& 8.6980e-2 & 6.0606e-6  &   2.2909e-3 & 2.1408e-4 
\\
$(\bP_2, P_1, P_2)$ & Nbc  & $10$ & 1.6249e-2 & 1.0196e-5 & 5.9745e-2 &   1.2259e-6 & 1.5956e-3  &  4.1399e-5
\\
$(\bP_2, P_1, P_2)$ & Nbc & $100$ & 1.1823e-2 & 1.0200e-5 & 3.1815e-2  &    3.5752e-7 & 8.9234e-4 &  1.2709e-5
\\
$(\bP_2, P_1, P_2)$ & Nbc  & $1000$ & 4.5587e-3 & 1.0206e-5 & 7.6105e-3  &   1.0284e-7 & 2.3539e-4 & 3.4295e-6 
\\
$(\bP_2, P_1, P_2)$ & Nbc & $10000$ & 2.7661e-3 & 1.0209e-5 & 9.3819e-4 &   1.4003e-8& 3.0125e-5  & 4.4616e-7
\\
$(\bP_2, P_1, P_2)$ & Nbc  & $100000$ & 2.7216e-3 &  1.0210e-5 & 9.6461e-5 & 1.4570e-9  & 3.1164e-6  & 4.6075e-8 
\\
\\
$(\bP_2, P_0, P_2)$ & Nbc  & $0$ & 2.8755e-2 & - - - - & 2.8909e-3  & - - - - & 0.1117 & - - - -   
\\
$(\bP_2, P_0, P_2)$ & Nbc  & $1$ & 2.4101e-2 & 1.3424e-5 &   1.7096e-2 & 1.8501e-4 & 1.0071  & 8.3547e-3 
\\
$(\bP_2, P_0, P_2)$ & Nbc  & $10$ & 1.6891e-2 & 1.1728e-5 &  5.3273e-3 & 5.1691e-5  & 0.7036 & 2.0768e-3
\\
$(\bP_2, P_0, P_2)$ & Nbc  & $100$ & 1.1847e-2 & 1.0845e-5 &  1.0918e-3 & 1.5535e-5 & 0.1865 & 5.7542e-4
\\
$(\bP_2, P_0, P_2)$ & Nbc  & $1000$ & 4.5718e-3 & 1.0236e-5 &   2.4468e-4 & 4.1319e-6 & 3.7840e-2 & 1.3773e-4
\\
$(\bP_2, P_0, P_2)$ & Nbc  & $10000$ & 2.7663e-3 & 1.0205e-5 & 3.0921e-5 & 5.2044e-7 & 7.9098e-3  & 1.6578e-5
\\
$(\bP_2, P_0, P_2)$ & Nbc & $100000$ & 2.7216e-3 & 1.0210e-5 & 3.1948e-6 & 5.3491e-8 & 9.6458e-4 &  1.6932e-6
\\
\\
$(\bP1b, P_1, P1b)$ & Nbc  & $0$ & 0.4637 &  - - - - & 0.3103  & - - - -  &   13.3382 & - - - -  
\\
$(\bP1b, P_1, P1b)$ & Nbc & $1$ & 0.3714 & 4.0602e-4 & 0.222  & 2.2436e-3  & 9.1670  & 4.7410e-2  
\\
$(\bP1b, P_1, P1b)$ & Nbc  & $10$ & 0.2325  &  4.1253e-4 & 8.4390e-2  &   5.8711e-4 & 3.5145  & 1.1282e-2 
\\
$(\bP1b, P_1, P1b)$ & Nbc  & $100$ & 0.1511 &  4.1728e-4 & 3.4756e-2  &    1.6482e-4 & 1.4964 & 3.5891e-3 
\\
$(\bP1b, P_1, P1b)$ & Nbc  & $1000$ & 5.6745e-2 & 4.3021e-4 & 9.8919e-3 & 6.6669e-5 & 0.3675 & 8.2588e-4 
\\
$(\bP1b, P_1, P1b)$ & Nbc  & $10000$ & 1.4120e-2 & 4.4175e-4 & 1.3482e-3 &   1.6045e-5& 4.3821e-2 & 1.6371e-5
\\
$(\bP1b, P_1, P1b)$ & Nbc  & $100000$ & 1.1276e-2 & 4.4370e-4  & 1.4028e-4  & 2.3152e-6  & 4.4671e-3  &  1.6635e-6
\\
\hline
\hline
\end{tabular}
\end{adjustbox}
\end{table}

\subsection{Marsigli Experiment}
This numerical experiment tests the proposed algorithm and reveals its effectiveness on a physical situation, which was described by Marsigli in 1681. This physical situation demonstrates that when two fluids with different densities meet, a motion driven by the gravitational force is created: the fluid with higher density rises over the lower one. Since the density differences can be modelled by the temperature differences with the help of the Boussinesq approximation, this physical problem is modelled by the incompressible Boussinesq system \eqref{Bouss} studied herein. \\ 
In the problem set-up, we follow the paper \cite{JLW03} of H. Johnston et al. The flow region taken is an insulated box $[0,8]\times[0,1]$ divided at $x=4$. The initial velocity is taken to be zero since the flow is at rest, and the initial temperature on the left hand side of the box is $\theta_0 = 1.5$, and on the right hand side $\theta_0=1.0$. The dimensionless flow parameters are set to be $Re=1,000, \, Ri =4, \, Pr=1$, and the flow starts from rest.\\ 
The first results we present are the direct numerical simulations (DNS) of the Boussinesq equations. We use finite element spaces $(\bP_2, P_1, P_2)$ for the velocity/pressure/temperature, respectively, on a finer, unstructed mesh, which provides $135,642$ velocity dof, $17,111$ pressure dof and $67,821$ temperature dof. All solutions are computed at $T = 2, 4, 8,$ taking a time step $\Delta t=0.025$. Our goal is to compare our scheme with the BE-FE method (i.e. no stabilization), and standard grad-div method with stabilization parameter $1.0$ on this physical problem on coarser meshes then is required by a DNS. In order to realize this aim, these three schemes are solved on the same moderately fine mesh, which gives $26,082$ velocity dof, $3,321$ pressure dof and $13,041$ temperature dof. We imposed homogeneous Dirichlet boundary conditions for the velocity and the adiabatic boundary condition for the temperature, and used $(\bP_2, P_1, P_2)$ for the velocity/pressure/temperature finite element spaces, respectively. All solutions are calculated at $T=2, 4, 8$ taking a time step $\Delta t=0.025$ with the same flow parameters as the DNS. The results are presented in Figures \ref{Coarse1}, \ref{Coarse2} and \ref{Coarse3}. It can be clearly seen that the modular grad-div method catches very well the flow pattern and temperature distribution of the DNS at each time level. Also it gives very similar results to the standard grad-div method (with parameter $1.0$) at each time level. However, the non-stabilized solution creates very poor solutions, and significant oscillations build in temperature and velocity as time progresses. 
\begin{figure}[h!]
\begin{center}
DNS $(T=2)$\\
\vspace{0.05cm}
\includegraphics[width=0.6\textwidth, height=2.25cm]{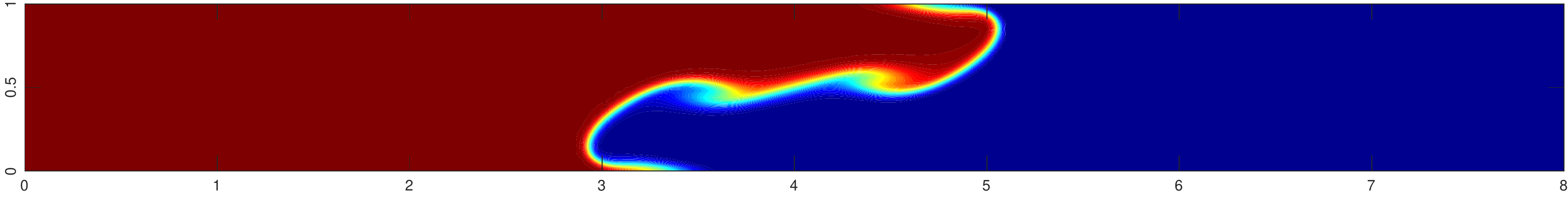}
\includegraphics[width=0.6\textwidth, height=2.25cm]{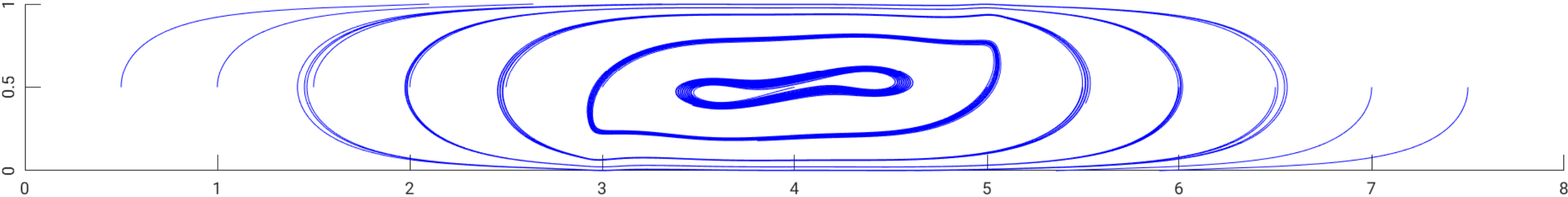}
\\
\vspace{0.1cm}
DNS $(T=4)$\\
\vspace{0.05cm}
\includegraphics[width=0.6\textwidth, height=2.25cm]{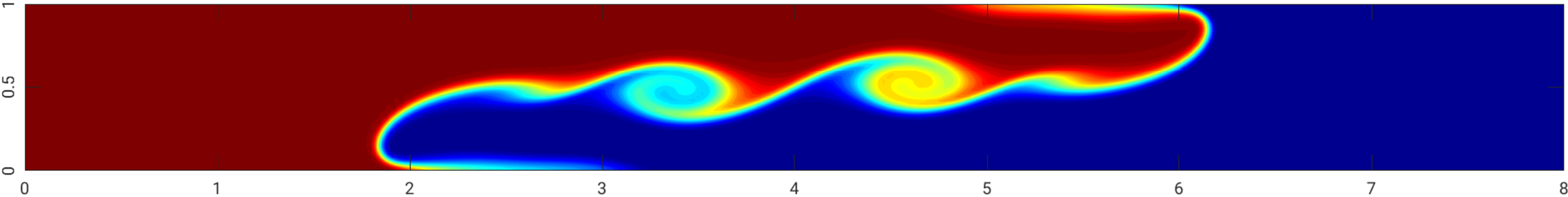}
\includegraphics[width=0.6\textwidth, height=2.25cm]{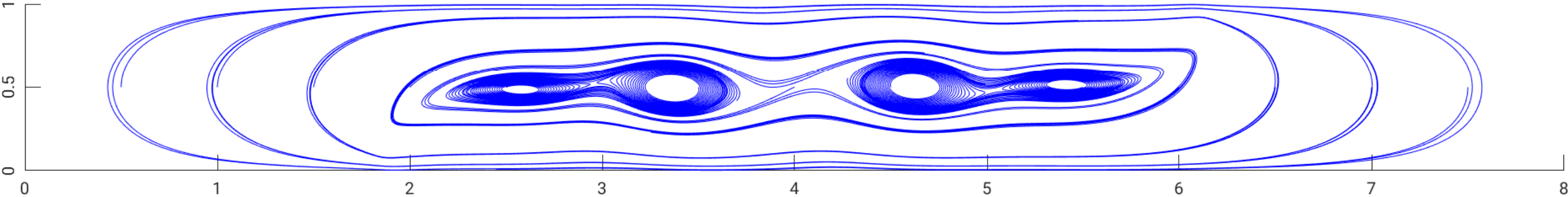}
\\
\vspace{0.1cm}
DNS $(T=8)$\\
\vspace{0.05cm}
\includegraphics[width=0.6\textwidth, height=2.25cm]{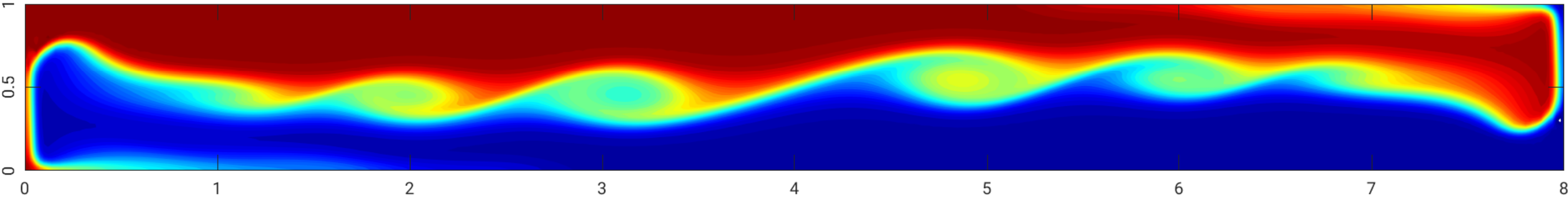}
\includegraphics[width=0.6\textwidth, height=2.25cm]{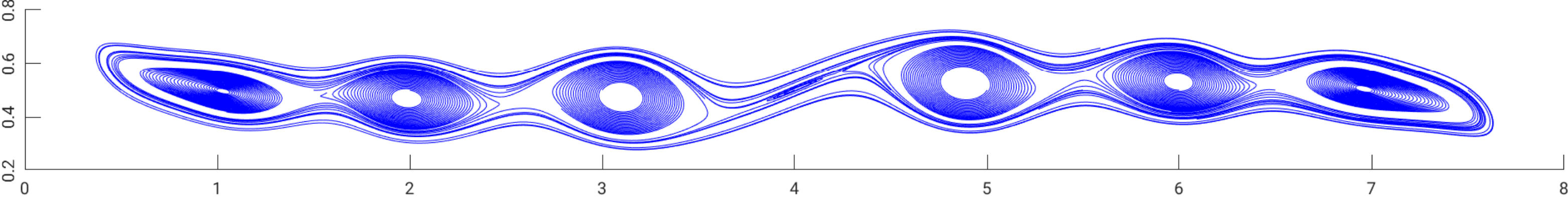}
\end{center}
\caption{The resolved temperature contours and velocity streamlines, from a fine mesh computation at end times $T=2,4,8$ with $\Delta t=0.025,$ $Re = 1,000, Pr =1$, and $ Ri = 4$.}
\label{DNS-BDF2}
\end{figure}
\begin{figure}[h!]
\begin{center}
No-stabilization $(T=2)$\\
\vspace{0.05cm}
\includegraphics[width=0.6\textwidth, height=2.25cm]{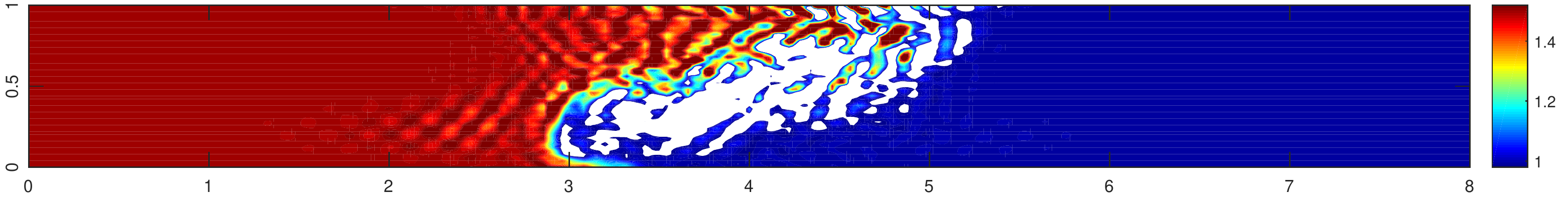}
\\
\vspace{0.2cm}
\includegraphics[width=0.6\textwidth, height=2.25cm]{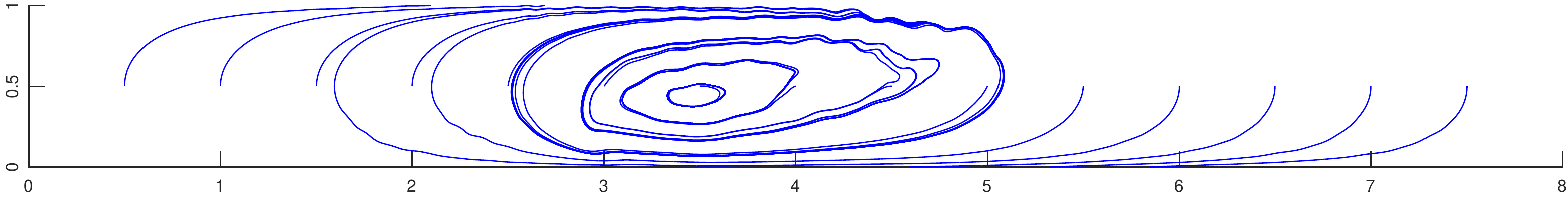}
\\
Usual grad-div stabilization $(T=2)$\\
\vspace{0.2cm}
\includegraphics[width=0.75\textwidth, height=2.25cm]{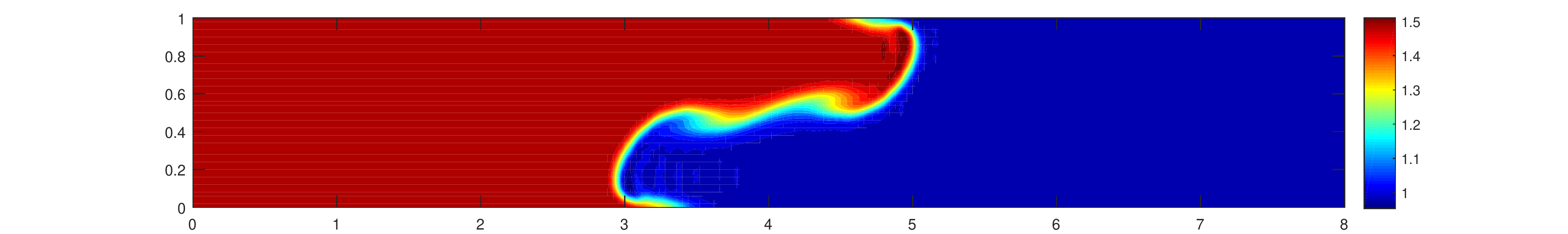}
\\
\vspace{0.2cm}
\includegraphics[width=0.75\textwidth, height=2.25cm]{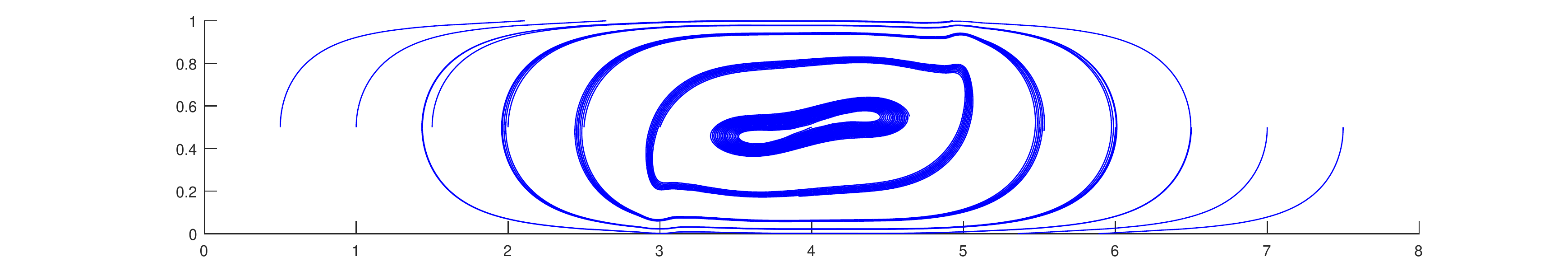}
\\
Modular grad-div stabilization $(T=2)$\\
\vspace{0.2cm}
\includegraphics[width=0.75\textwidth, height=2.25cm]{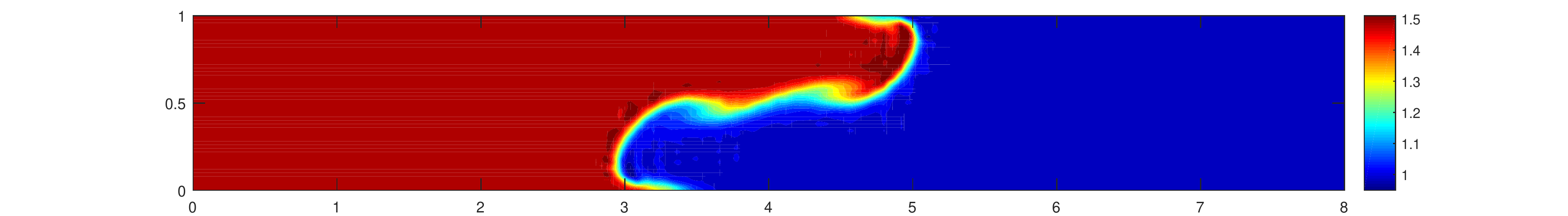}
\\
\vspace{0.2cm}
\includegraphics[width=0.75\textwidth, height=2.25cm]{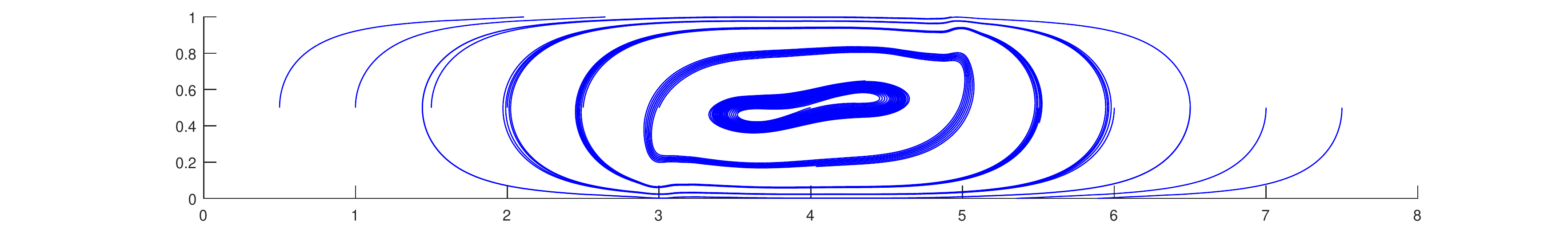}
\\
\vspace{0.1cm}
\end{center}
\caption{The temperature contours and velocity streamlines of BE-FEM (no-stabilization), the usual grad-div, and the modular grad-div, respectively, from a coarse mesh computation at $T=2$ with $\Delta t=0.025$, $Re = 1,000, Pr =1$, and $Ri = 4$.}
\label{Coarse1}
\end{figure}

\begin{figure}[h!]
\begin{center}
No-stabilization $(T=4)$\\
\vspace{0.05cm}
\includegraphics[width=0.6\textwidth, height=2.25cm]{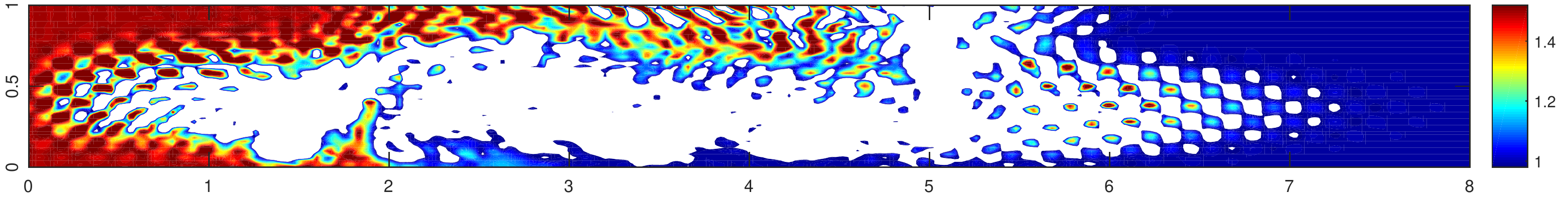}
\\
\vspace{0.2cm}
\includegraphics[width=0.6\textwidth, height=2.25cm]{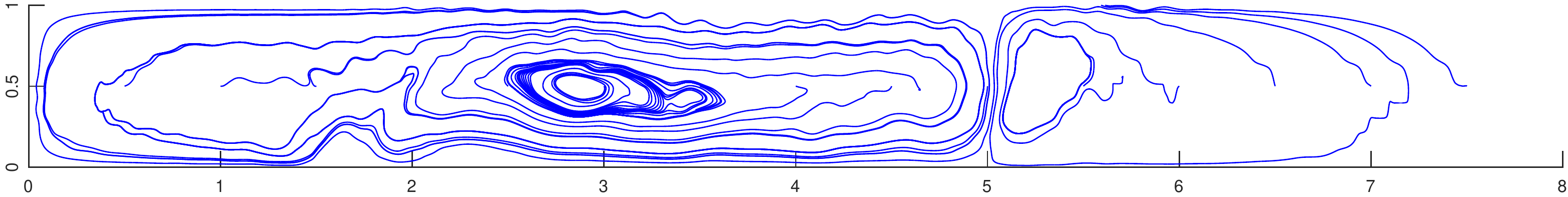}
\\
Usual grad-div stabilization $(T=4)$\\
\vspace{0.2cm}
\includegraphics[width=0.75\textwidth, height=2.25cm]{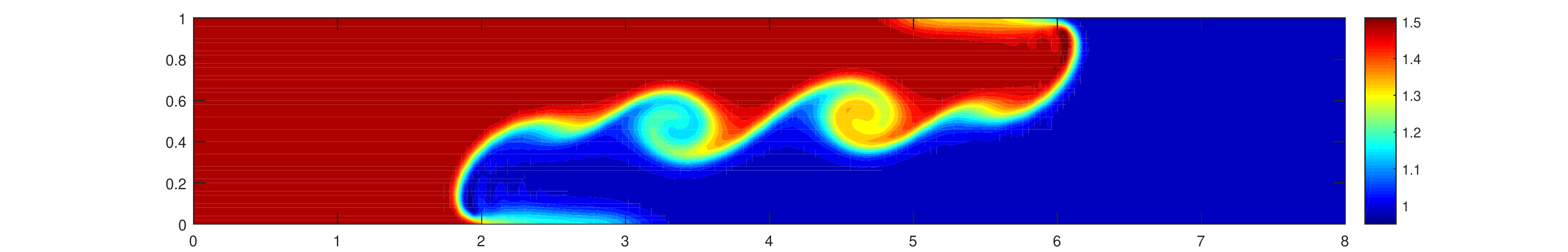}
\\
\vspace{0.2cm}
\includegraphics[width=0.75\textwidth, height=2.25cm]{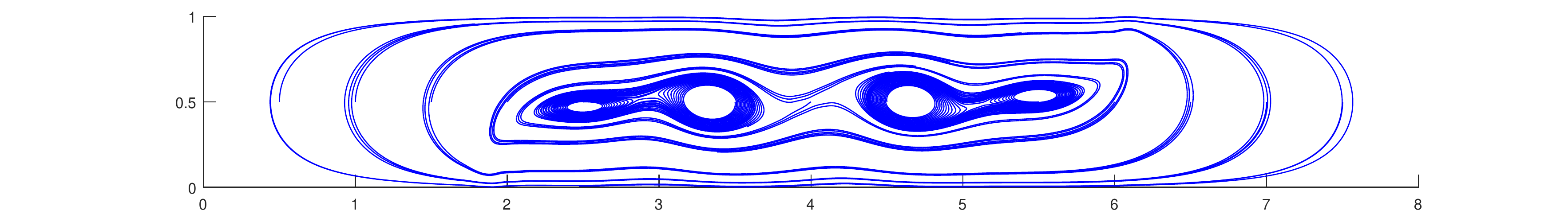}
\\
Modular grad-div stabilization $(T=4)$\\
\vspace{0.2cm}
\includegraphics[width=0.75\textwidth, height=2.25cm]{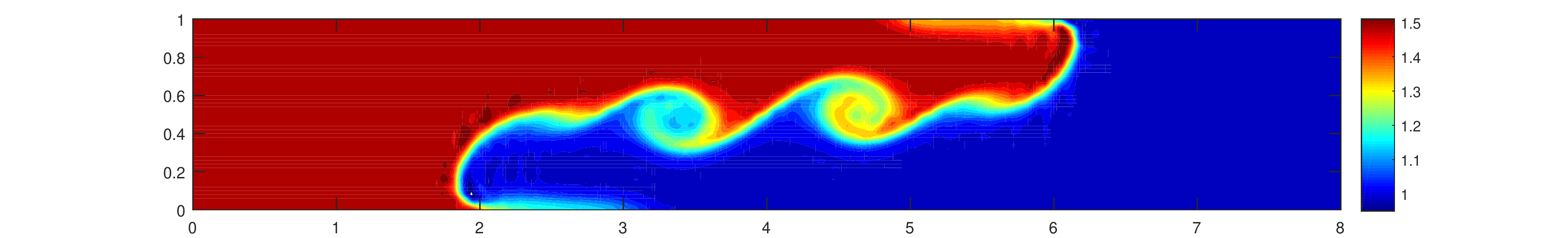}
\\
\vspace{0.2cm}
\includegraphics[width=0.75\textwidth, height=2.25cm]{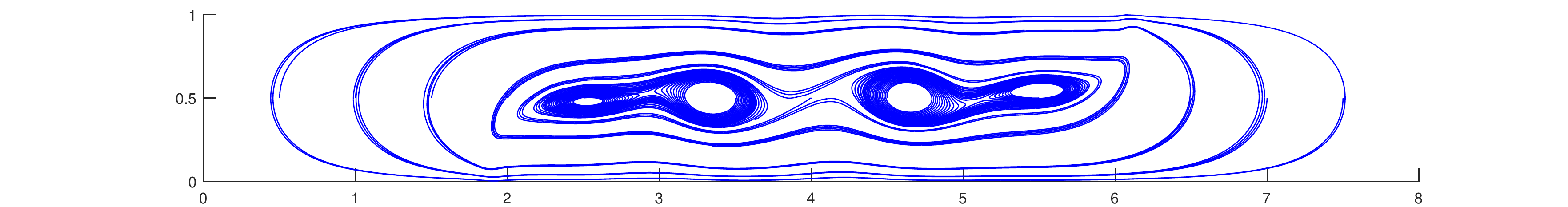}
\\
\vspace{0.1cm}
\end{center}
\caption{The temperature contours and velocity streamlines of BE-FEM (no-stabilization), the usual grad-div, and the modular grad-div, respectively, from a coarse mesh computation at $T=4$ with $\Delta t=0.025$, $Re = 1,000, Pr =1$, and $Ri = 4$.}
\label{Coarse2}
\end{figure}
\begin{figure}[h!]
\begin{center}
No-stabilization $(T=8)$\\
\vspace{0.05cm}
\includegraphics[width=0.6\textwidth, height=2.25cm]{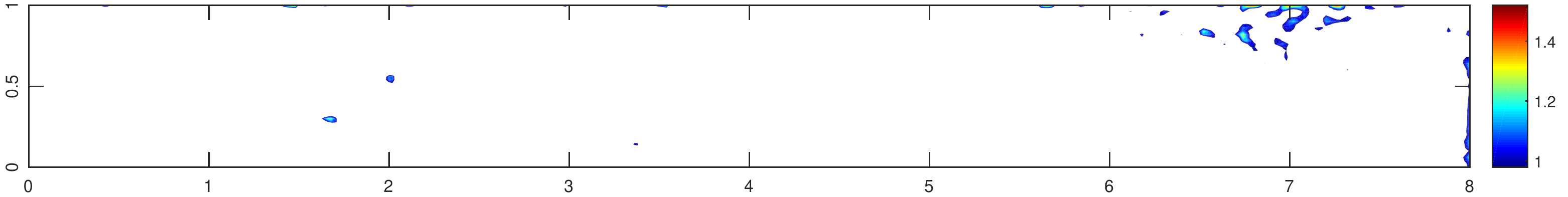}
\\
\vspace{0.2cm}
\includegraphics[width=0.6\textwidth, height=2.25cm]{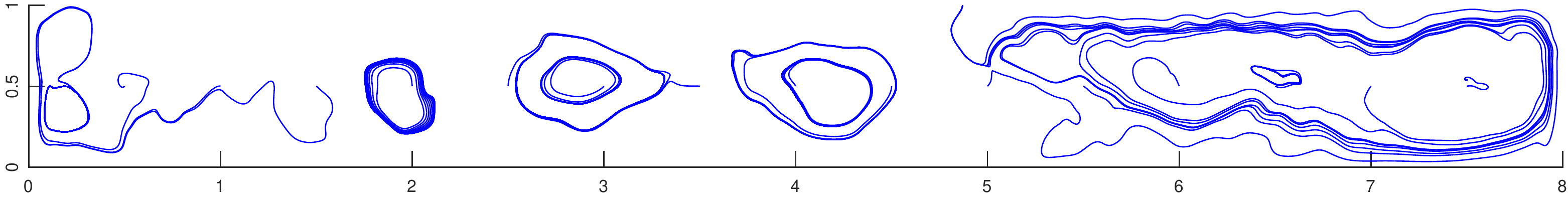}
\\
Usual grad-div stabilization $(T=8)$\\
\vspace{0.2cm}
\includegraphics[width=0.75\textwidth, height=2.25cm]{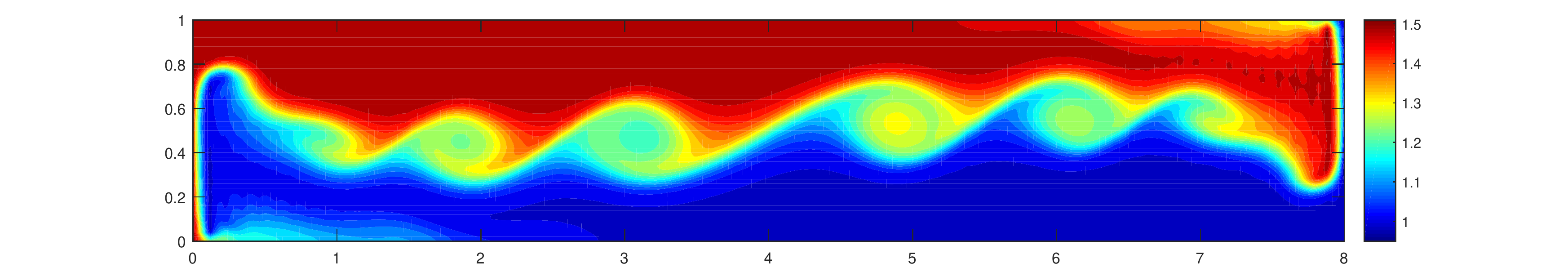}
\\
\vspace{0.2cm}
\includegraphics[width=0.75\textwidth, height=2.25cm]{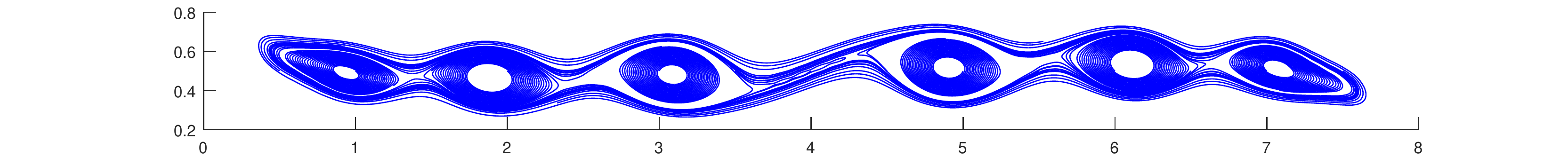}
\\
Modular grad-div stabilization $(T=8)$\\
\vspace{0.2cm}
\includegraphics[width=0.75\textwidth, height=2.25cm]{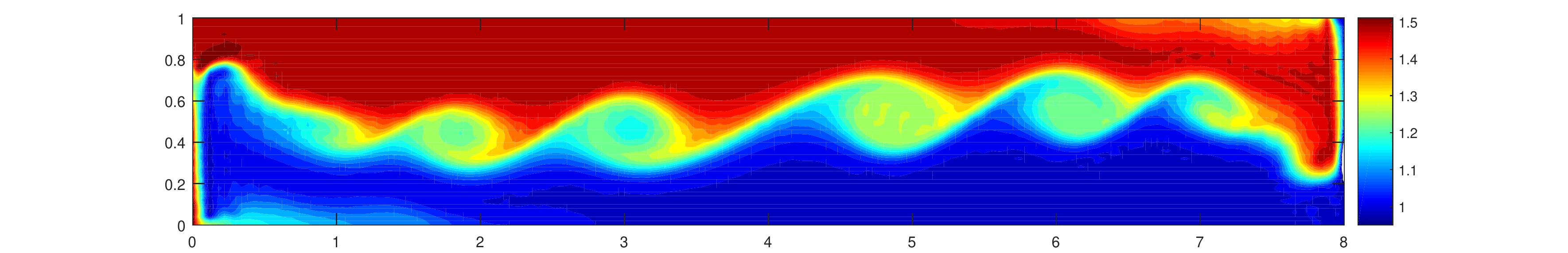}
\\
\vspace{0.2cm}
\includegraphics[width=0.75\textwidth, height=2.25cm]{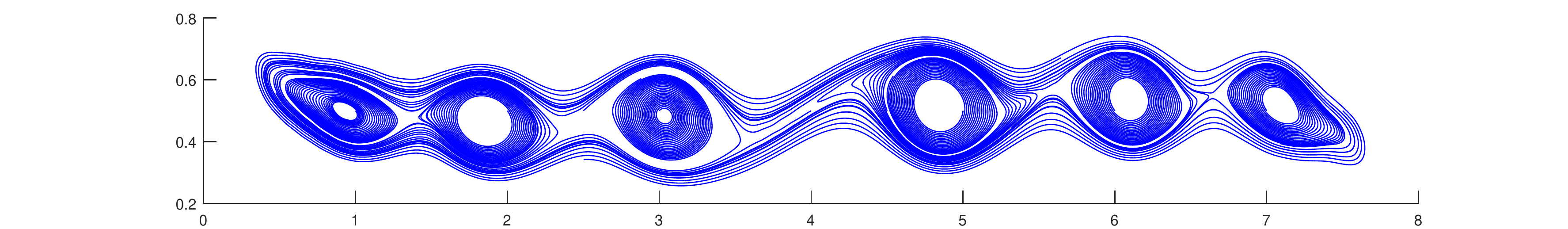}
\\
\vspace{0.1cm}
\end{center}
\caption{The temperature contours and velocity streamlines of BE-FEM (no-stabilization), the usual grad-div, and the modular grad-div, respectively, from a coarse mesh computation at $T=8$ with $\Delta t=0.025$, $Re = 1,000, Pr =1$, and $Ri = 4$.}
\label{Coarse3}
\end{figure}

\section{Conclusions}
This paper proposed, analyzed, and tested modular grad-div stabilization methods in discretization of the Boussinesq flows. Unconditional stability and convergence results are established for the system. Numerical experiments were given that verified the convergence rates derived from finite element error analysis. Also, the reliability and  efficiency of the methods were tested with some numerical experiments. These results reveal that the methods are very accurate when compared to the non-stabilized methods, and have effects on solutions similar to that of standard grad-div stabilization.

\end{document}